\documentclass[12pt, pdftex]{amsart}
\usepackage[margin=1in]{geometry}
\usepackage[utf8]{inputenc}
\usepackage{tikz}
\usetikzlibrary{arrows.meta,decorations.pathreplacing,decorations.markings,shapes,calc}
\usetikzlibrary{matrix,arrows,decorations.pathmorphing,backgrounds,decorations.markings,positioning}
\usepackage{graphics}

\usepackage[all]{xy}
\usepackage{comment}
\usepackage{amsmath}
\usepackage{amssymb}
\usepackage{amsthm}
\usepackage{here}
\usepackage{amscd} 
\usepackage{tikz-cd}
\usepackage{mathrsfs}
\usepackage{mathtools}
\usepackage{mathabx}
\usepackage{scalefnt}
\usepackage{url}
\usepackage{breqn}
\theoremstyle{plain}
\newtheorem{thm}{Theorem}[section]
\newtheorem{lem}[thm]{Lemma}
\newtheorem{cor}[thm]{Corollary}
\newtheorem{prop}[thm]{Proposition}

\newtheorem{redm}[thm]{Reduction Method}
\theoremstyle{definition}

\newtheorem{rem}[thm]{Remark}

\newtheorem{defn}[thm]{Definition}
\newtheorem{prob}[thm]{Problem}

\newtheorem{ex}[thm]{Example}

\numberwithin{equation}{section}
\def\A{{\mathbb A}}

\def\Q{{\mathbb Q}}

\def\Z{{\mathbb Z}}
\def\C{{\mathbb C}}
\def\P{{\mathbb P}}
\def\B{{\mathbb B}}

\def\SL{\mathop{\mathrm{SL}}\nolimits}

\def\dim{\mathop{\mathrm{dim}}\nolimits}

\def\Stab{\mathop{\mathrm{Stab}}\nolimits}

\def\tor{\mathop{\mathrm{T}}}
\def\BB{\mathop{\mathrm{BB}}}

\def\G{\mathrm{G}}
\def\K{\mathrm{K}}
\def\T{\mathrm{T}}
\def\BB{\mathrm{B}}

\def\L{\mathscr{L}}

\def\M{\mathcal{M}}

\def\OO{\mathscr{O}}

\def\A{\mathcal{A}}

\def\Sp{\mathop{\mathrm{Sp}}\nolimits}

\allowdisplaybreaks[4]

\newcommand{\defeq}{\vcentcolon=}

\allowdisplaybreaks[1]

\begin{document}

\title[The Universe of Deligne--Mostow Varieties]{The Universe of Deligne--Mostow Varieties}
\author[Klaus Hulek and Yota Maeda]{Klaus Hulek$^{1}$ \and Yota Maeda$^{2, 3}$}
\email{hulek@math.uni-hannover.de\\ y.maeda.math@gmail.com}

\date{\today}

\maketitle
\vspace{-1em}
\begin{center}
  \begin{minipage}{0.9\textwidth}
    \centering
    {\small
    $^{1}$ Institut für Algebraische Geometrie, Leibniz University Hannover, Germany.\\
    $^{2}$ Fachbereich Mathematik, Technische Universität Darmstadt, Germany.\\
    $^{3}$ Mathematical Institute, Tohoku University, Japan.\\[0.5ex]
    }
  \end{minipage}
\end{center}

\vspace{1.5em}

\begin{abstract}
Deligne and Mostow investigated period maps on the configuration spaces $M_{0,n}$ of $n$ ordered points on $\P^1$.
The images of these maps are open subsets of certain ball quotients. Moreover, they extend to isomorphisms between GIT-quotients and the Baily-Borel compactifications.
Building on a theorem of Gallardo, Kerr and Schaffler, the period maps lift to isomorphisms between two natural compactifications, namely the Kirwan blow-up and the toroidal compactification. In this paper, we look at the more general situation where we also allow unordered or partially ordered $n$-tuples. Our main result is an easily verifiable criterion that, in this broader setting, determines when the Deligne-Mostow period maps still lift to isomorphisms between the Kirwan blow-up and the toroidal compactification.  
We further investigate a partial ordering among Deligne-Mostow varieties, which reduces this problem to considering minimal or maximal Deligne-Mostow varieties with respect to this partial ordering. 
As a byproduct, we prove that, in general, Kirwan's resolution pair is not a log canonical log minimal model and not log $K$-equivalent to the unique toroidal compactification.
\end{abstract}

\section{Introduction}
\subsection{Background of period maps}
Let $\M$ be a moduli space and $D/\Gamma$ be a modular variety connected to $\M$ via a period map,
where $D$ is a Hermitian symmetric domain and $\Gamma$ is an arithmetic subgroup.
A fundamental problem is:
\begin{prob}
\label{prob:extendability}
    When does a period map 
    \[\M\to D/\Gamma\]
    extend to a morphism between suitable algebraic  compactifications
        \[\overline{\M}\to \overline{D/\Gamma}?\]
\end{prob}
There are a number of related deep studies in this topic: semi-toroidal compactifications \cite{Sha80, Loo85, Loo86, Loo03a, Loo03b}, KSBA compactifications \cite{Laz16, MS21} and K-stability \cite{ABE22, AE23}, solely in the case of moduli spaces of K3 surfaces.
Recently, Alexeev, Engel, Odaka, et. al. \cite{AE23, AEH21, Oda22} gave an interpretation of semi-toroidal compactifications of $D/\Gamma$ in terms of $K$-stability and log minimal model program (LMMP) (which we 
will use in Section \ref{section:applications}).
Whether period maps can be extended to suitable compactifications depends on the situation, in particular the choice of the compactification on both sides.
This in turn is closely related to the moduli interpretation of the boundary $\overline{\M}\setminus\M$.

The very first example that comes to mind is the following well-known case:  there is an isomorphism from the set of isomorphism classes of elliptic curves (which carries in a natural way the structure of a Riemann surface)  to (the open subset of) the modular curve 
\begin{align}
\label{mor:ex_modular_curve}
    \M_{1,1}\stackrel{\sim}{\longrightarrow}Y(1)\defeq\mathbb{H}/\SL_2(\Z) (\cong \C).
\end{align}
The left-hand side has a compactification, which is known as the Deligne-Mumford compactification $\overline{\M}_{1,1}$, adding a point to $\M_{1,1}$ corresponding to the nodal curve $y^2=x^3-xy$.
In this case, the right-hand side only allows one compactification (as a Riemann surface), namely $\P^1$.
The isomorphism (\ref{mor:ex_modular_curve}) can be extended to 
\[\overline{\M}_{1,1}\stackrel{\sim}{\longrightarrow}\overline{Y(1)} (\cong \P^1),\]
where $\overline{Y(1)}=Y(1)\cup\{\infty\}$ is the level 1 modular curve.
Hence, we obtain a positive answer to Problem \ref{prob:extendability} for the moduli spaces of elliptic curves.

But this is a very simple case and in general the situation is much more complex and subtle. To give an example,  we consider the Torelli map:
\begin{align}
\label{mor:Torelli_map}
    \M_g&\to \A_g\defeq\mathbb{H}_g/\Sp_{2g}(\Z)\\
    C&\mapsto \mathrm{Jac}(C).\notag
\end{align}
There is a very natural compactification on the left-hand side, namely the Deligne-Mumford compactification $\overline{\M_g}$ (generalizing $\overline{\M}_{1,1}$).
Unlike in the case of modular curves, there are several ways to compactify the Siegel modular varieties $\A_g$ algebraically.
There is the minimal compactification, which is the Baily-Borel compactification $\overline{\A_g}^{\BB}$.
Its boundary $\overline{\A_g}^{\BB}\setminus\A_g$ consists of the low-dimensional Siegel modular varieties.
By taking Jacobians of the normalization of a curve in $\overline{\M_g}\setminus\M_g$, one can show that the Torelli map (\ref{mor:Torelli_map}) can be extended to a morphism
\[\overline{\M_g}\to\overline{\A_g}^{\BB}.\]
From the point of view of moduli, this map loses a lot of geometric information. 
Now, there are other choices to compactify $\A_g$ algebraically, which are known as toroidal compactifications $\overline{\A_g}^{\T, \tau}$.
These compactifications are constructed via toric geometry and hence depend on the datum of a fan $\tau$, which determines a cone decomposition.
There are several ways to choose a fan $\tau$, and the three standard ones are
\begin{enumerate}
    \item $\tau^{\mathrm{perf}}$: 1st Voronoi fan 
    \item $\tau^{\mathrm{vor}}$: 2nd Voronoi fan
    \item $\tau^{\mathrm{cent}}$: central cones;
\end{enumerate}
see \cite[Introduction]{AB12} for the details.
Mumford and Namikawa \cite{Nam76a, Nam76b} proved that the Torelli map (\ref{mor:Torelli_map}) can be extended for $\tau=\tau^{\mathrm{vor}}$.
In addition, Alexeev and Brunyate \cite{AB12} obtained an analogous positive result for $\tau=\tau^{\mathrm{perf}}$.
However, it is known that the situation for $\tau^{\mathrm{cent}}$ is different.
Though there is a regular map 
\[\overline{\M_g}\to\overline{\A_g}^{\T, \tau^{\mathrm{cent}}}\]
for $g\leq 8$ by Namikawa \cite{Nam73} and Alexeev et. al \cite{ALT+12}, Alexeev and Brunyate \cite{AB12} proved that the map cannot be  extended for $g\geq 9$.
These results show that the situation can be very subtle and one cannot expect a uniform answer to problem \ref{prob:extendability}.

Here we shall study the period maps constructed by Deligne and Mostow \cite{DM86, Mos86} on the configuration space $M_w\defeq M_{0,n}$ of ordered $n$-tuples on $\P^1$ 
and its quotient by certain symmetric groups $S[w]$ defined below.

Let $n\geq 5$ be a positive integer and $w=(w_1,\cdots, w_n)\in\Q_+^n$ be an $n$-tuple of weights with $\sum_i w_i = 2$.
We say that a weight vector $w$ satisfies \textit{condition INT} if 
\[(1-w_i-w_j)^{-1} \in \Z\]
for any pair $(w_i,w_j)$ satisfying $w_i+w_j<1$ and $i\neq j$.
For any such  pair $(n,w)$, Deligne and Mostow \cite{DM86} constructed a period map
\begin{align}
\label{mor:ordered_period_map}
    \tag{INT}\ M_w&\hookrightarrow\B^{n-3}/\Gamma_w.
    \end{align}
    through the monodromy of the hypergeometric differential forms on curves.
Here, $\B^{n-3}$ is the $(n-3)$-dimensional ball (type I domain) acted on by a  (not necessarily arithmetic) unitary group $\Gamma_w$.
The INT condition can be relaxed and generalized as follows.
Let $S$ be a nonempty subset of $\mathbb{N}_n\defeq\{1,\cdots, n\}$.
We say that $w$ satisfies the \textit{$\Sigma$INT condition for $S$} if $w_i=w_j$ for  $i,j\in S$ and 
\[(1-w_i-w_j)^{-1}\]
is an element of $\Z$ (resp. $1/2\Z$) if $i\not\in S$ or $j\not\in S$ (resp. $i,j\in S$) for any pair $(w_i,w_j)$ satisfying $w_i+w_j<1$ and $i\neq j$.
For such an $S$, we denote by $S[w]$ the symmetric group $\Sigma_{|S|}$ acting on $S$ and $w(S)\defeq w_i\in\Q$ the component of the weight vector $w$ corresponding to any $i$ contained in $S$.
In this case, there exists, according to \cite{Mos86}, a unitary group $\Gamma_{w,S}$ and  a period map
\begin{align}
    \label{mor:unordered_period_map}
    \tag{$\Sigma$INT-$S$}\ M_{w,S}\defeq M_{w}/S[w]&\hookrightarrow\B^{n-3}/\Gamma_{w,S}.
\end{align}
 One can check that the INT condition is equivalent to the $\Sigma$INT-$S$ condition for a one-point set $S$.
 In this way, in order to treat both cases in a uniform way, we consider $\Gamma_w$ and $M_{w}$ as a special case of $\Gamma_{w,S}$ and $M_{w,S}$. 
 In other words, if $S$ consists of one point, then we are in the situation of the INT condition and we simply use the notation $\Gamma_w$ and $M_w$ instead of $\Gamma_{w,S}$ and $M_{w,S}$.  
Note that since we consider the partition $\mathbb{N}_n = S \sqcup S^{c}$, the symmetry groups we consider are of the form $S[w] = \Sigma_{|S|}$. 
 
We call $(w,S)$ a \textit{Deligne-Mostow pair} if 
\begin{itemize}
    \item $w$ satisfies $\Sigma$INT-$S$,
    \item $\Gamma_{w,S}$ is a non-cocompact arithmetic subgroup.
\end{itemize}
We refer to  \cite[Tables 2, 3]{GKS21} for a complete list of the possible weights.
It is natural and geometrically interesting to study Deligne-Mostow theory in its most general setting possible, namely all cases of (partially) unordered sets which Deligne-Mostow theory can handle.
For this, we consider the following equivalence relation.
Two Deligne-Mostow pairs $(w,S)$ and $(w',S')$ are defined to be equivalent if and only if the length of $w$ and $w'$ coincide, denoted by $n$ here, and there is an 
element $g$ of the symmetric group $\Sigma_n$ so that $g(w) = w'$ and $g(S) = S'$.
We denote by $A_{\mathrm{DM}}$ the equivalence class of the set of all Deligne-Mostow pairs and call this the \textit{universe of the Deligne-Mostow varieties}.
Note that the $\Sigma$INT-$S$ condition in \cite{Mos86} is defined for a partition of $\mathbb{N}_n$ into $S$ and its complement $S^c$ and thus the symmetry group is of the form $S[w] = \Sigma_{|S|}$. 
A prior, it is also possible to consider a decomposition of $S$ into more than two subsets, resulting in 
symmetry groups such as $\Sigma_{r_1}\times\Sigma_{r_2}$ acting on two or more connected components. But this is not in the realm of Mostow's work, and therefore not considered in this paper, 
see also Remark \ref{rem:DengGallardo}.

As we will recall in Theorem \ref{thm:DM_BB}, the morphisms (\ref{mor:ordered_period_map}) and (\ref{mor:unordered_period_map}) can be compactified to give isomorphisms  
between GIT-quotients and the Baily-Borel compactifications.
As we shall further discuss below, it follows from the work of Gallardo, Kerr, and Schaffler, combined with results of Kiem and Moon, 
that the period map (\ref{mor:ordered_period_map}) can be extended to the Kirwan compactification of  $M_{w}$,  providing an isomorphism to the 
toroidal compactification; see Theorem \ref{thm:GKS21}.

In this paper, we shall investigate the partially unordered situation, namely when the morphism (\ref{mor:unordered_period_map}) can be extended to the toroidal compactification. A complete answer is given in Theorem \ref{mainthm:main_extendability}.

\subsection{Main results}
\label{subsection:main_results}
Deligne and Mostow proved the following:
\begin{thm}[{\cite{DM86, Mos86}}]
\label{thm:DM_BB}
    Let $(w, S)\in A_{\mathrm{DM}}$ be a Deligne-Mostow pair.
    Then the open immersion \eqref{mor:unordered_period_map},
   the complement of whose image is the discriminant divisor,
    can be extended to the GIT and Baily-Borel compactifications: 
     \begin{align}  
  \tag{C$\Sigma$INT-$S$}
 \label{isom:compactified_sigma_int}
M_{w,S}^{\G}\defeq(\P^1)^n/\!/_wS[w]\times\SL_2(\C)&\cong\overline{\B^{n-3}/\Gamma_{w,S}}^{\BB}=:X_{w,S}^{\BB}.
    \end{align}
\end{thm}    
If (INT) holds and $S$ is taken to be a point, then we drop the symbol $S$ from our notation and simply write    
 \begin{align}
       \tag{CINT} 
\label{isom:compactified_int}
M_w^{\G}\defeq(\P^1)^n/\!/_w\SL_2(\C)&\cong\overline{\B^{n-3}/\Gamma_w}^{\BB}=:X_w^{\BB}.
\end{align}

 For the definitions and properties of GIT problems on the configuration of points in $\P^1$, see \cite[Chapter 11]{Dol03} or \cite[Chapter 7]{Muk12}.
Here, the symbol ``C" means ``Compactified".
There are other choices for the compactification of these spaces.
Gallardo, Kerr and Schaffler clarified the relationship between two natural blow-ups of $M_w^{\G}\cong X_w^{\BB}$, namely the Kirwan blow-up $M_w^{\K}$ and the canonical toroidal compactification $X_w^{\T}$. The centres of these blow-ups are the polystable points and the cusps respectively, which are identified under the isomorphism (\ref{isom:compactified_int}).
We say that two compactifications of a Deligne-Mostow variety $M_{w,S}$ and the corresponding  ball quotient $\B^{n-3}/\Gamma_{w,S}$ respectively,  are \textit{naturally isomorphic} if 
the Deligne-Mostow isomorphism (\ref{mor:unordered_period_map}) from $M_{w,S}$ onto its image in  $\B^{n-3}/\Gamma_{w,S}$ extends to these compactifications.  
\begin{thm}[{\cite[Theorem 1.1]{GKS21}}]
\label{thm:GKS21}
    The Deligne-Mostow period map \eqref{mor:ordered_period_map}
    \[M_w \hookrightarrow X_w\defeq\B^{n-3}/\Gamma_w\]
    extends to an isomorphism
    \[M_w^{\K}\stackrel{\sim}{\longrightarrow}X_w^{\T}.\]
    In other words, the Kirwan blow-up and the toroidal compactification are naturally isomorphic in the INT case.
\end{thm}

In this formulation of the theorem we used the fact that the moduli spaces of weighted pointed curves, introduced in \cite{Has03}, and the Kirwan blow-up coincide.
This follows from \cite[Theorem 1.1]{KM11}; see also \cite[Remark 2.11]{GKS21}.
In this paper, we shall prove that an analogue for the spaces $M_{w,S}^{\K}$ and $X_{w,S}^{\BB}$ of Theorem \ref{thm:GKS21}  does \textit{not} hold for all Deligne-Mostow varieties listed 
in \cite[Tables 2, 3]{GKS21}, and we shall further give a criterion when the compactifications are naturally isomorphic.

Now, for a pair $(w,S) \in A_{\mathrm{DM}}$, let us introduce a numerical condition.
The proposition (T) is defined as the \textit{negation} of the following condition:
\begin{equation}\label{equ:NT}
\mathrm{There\ exist\ two\ sets\ }T_1\subset S,\ T_2\subset S^{c}\ \mathrm{so\ that}\ |T_1|\geq 3\ \mathrm{and}\ \sum_{i\in T_1\sqcup T_2}w_i = 1.
\end{equation}
Here the notation $(\mathrm{T})$  stands for {\em transversal}. The reason for this name will become clear in Section \ref{sec:mainresults}. 

\begin{thm}[{Theorem \ref{thm:main_extendability}}]
\label{mainthm:main_extendability}
    The (Deligne-)Mostow isomorphism \eqref{mor:unordered_period_map}
    \[M_{w,S}\hookrightarrow X_{w,S}\defeq\B^{n-3}/\Gamma_{w,S}\]
    extends to an isomorphism between $M^{\K}_{w,S}$ and $X_{w,S}^{\T}$ if and only if  $(\mathrm{T})$ holds.
Otherwise, neither the isomorphism given by \eqref{isom:compactified_sigma_int}, nor its inverse lift to $M^{\K}_{w,S}$ or $X_{w,S}^{\T}$ respectively. In particular, the Kirwan blow-up and the toroidal compactification are not naturally isomorphic in this case.
\end{thm}
Taking $S$ to be a one-point set, Theorem \ref{mainthm:main_extendability} becomes Theorem \ref{thm:GKS21}.

\begin{rem}
Some cases are known where analogous theorems hold: the moduli spaces of cubic surfaces \cite{CMGHL23}, 8 points in $\P^1$ \cite{HM25} and 12 points in $\P^1$ \cite{HKM24}.
Our result is based on and generalizes the results by \cite[Theorem 1.1]{HM25} and \cite[Theorem 1.4 (2)]{HKM24}.
\end{rem}
Theorem \ref{mainthm:main_extendability} is closely related to the question whether taking  Kirwan blow-ups is compatible with taking finite quotients. In the case of ball quotients, toroidal compactifications behave functorially
under taking finite quotients. This follows from the general theory of toroidal compactifications, see also \cite[Theorem 1.1]{GKS21}. The question is more subtle for Kirwan blow-ups. We shall discuss this in detail
in Subsection \ref{subsec:functKirwan}. 
\begin{cor}
Let $(w,S)\in A_{\mathrm{DM}}$ be a Deligne-Mostow pair.
    Then the $S[w]$-quotient $M_w^{\K}/S[w]$ of the Kirwan blow-up is isomorphic to the Kirwan blow-up $M_{w,S}^{\K}$ of $M_{w,S}^{\G} \cong M_w/S[w]$ if and only if $\mathrm{(T)}$ holds.
    In particular, if $|S| = 1, 2$, then we always have an isomorphism $M_w^{\K}/S[w] \cong M_{w,S}^{\K}$.
\end{cor}

We shall now explain how this relates to previous results.
First, the problem of whether the Kirwan blow-up and the toroidal compactification are isomorphic or not has been treated in several previous studies. This is summarized in Table \ref{tab:previous studies on isom}.
\renewcommand{\arraystretch}{1.5}
\begin{table}[h]
  \centering
  \caption{Previous studies on the classification of compactifications.}
   \label{tab:previous studies on isom}
\begin{tabular}{|c|c|c|c|c|}
   \hline
   Number fields & \multicolumn{2}{|c|}{Gaussian} &  \multicolumn{2}{|c|}{Eisenstein} \\ \hline
   Isomorphic ? &$M^{\K}_{w,S} \cong X_{w,S}^{\T}$ & $M^{\K}_{w,S} \not\cong X_{w,S}^{\T}$ & $M^{\K}_{w,S} \cong X_{w,S}^{\T}$ & $M^{\K}_{w,S} \not\cong X_{w,S}^{\T}$ \\ \hline
   $\#$ Cases & 6 INT cases  & $(w_{\G}, \mathbb{N}_8)$  & 7 INT cases  & $(w_{\mathbb{E}}, \mathbb{N}_{12})$\\ \hline
      References & \cite{GKS21} & \cite{HM25} & \cite{GKS21} & \cite{HKM24}\\ \hline
 \end{tabular}
\end{table}

Our aim is to study the universe $A_{\mathrm{DM}}$. A direct calculation shows that $A_{\mathrm{DM}}$ contains 85 elements, see Tables \ref{table:DM Gaussian}, \ref{table:DM Eisenstein}.
Theorem \ref{mainthm:main_extendability} states for which of these cases the Kirwan blow-up and the toroidal compactification are naturally isomorphic. One of the main results of this paper is 
that one can define a partial order $\prec$ which provides much insight into the structure of the universe $A_{\mathrm{DM}}$ and is very helpful in analyzing the behaviour of the compactifications. This goes back to Doran's work  \cite{Dor04, DDH18}. 
Consider the weights 
\begin{align*}
w_{\G} &\defeq \left(\frac{1}{4},\frac{1}{4},\frac{1}{4},\frac{1}{4},\frac{1}{4},\frac{1}{4},\frac{1}{4},\frac{1}{4}\right)\\
w_{\mathrm{E}} &\defeq \left(\frac{1}{6},\frac{1}{6},\frac{1}{6},\frac{1}{6},\frac{1}{6},\frac{1}{6},\frac{1}{6},\frac{1}{6},\frac{1}{6},\frac{1}{6},\frac{1}{6},\frac{1}{6}\right).
\end{align*}
In the first case condition INT is satisfied; in the second case it is not, but $\Sigma$INT-$S$ holds for the set $S=\mathbb{N}_{12}$.
These weights give rise to the Deligne-Mostow varieties $M_{w_{\mathrm{G}}}^{\G}$ and $M_{w_{\mathrm{E}, \mathbb{N}_{12}}}^{\G}$.
Doran \cite{Dor04, DDH18} showed that all Deligne-Mostow varieties $M_{w,S}^{\G}$ have a finite morphism to a sub-ball quotient of either $M_{w_{\mathrm{G}}}^{\G}$ or $M_{w_{\mathrm{E}, \mathbb{N}_{12}}}^{\G}$.
For this reason he called these varieties the Gaussian and the Eisenstein \textit{ancestral} Deligne-Mostow varieties.
We refer to Section \ref{sec:The moduli spaces of partly ordered points} for more details.

We extend the notion of ancestral varieties by introducing a partial order $\prec$, see Definition \ref{defn:weights}. 
Proposition \ref{prop:partial order}
shows that for two pairs $(w,S), (w,S') \in A_{\mathrm{DM}}$ the relation $(w,S)\prec (w,S')$ (Definition \ref{defn:weights}) is equivalent to the existence of a closed immersion of the corresponding GIT spaces
\[M_{w,S}^{\G} \hookrightarrow M_{w',S'}^{\G}.\]
The ancestral varieties $(w_{\G}, \mathbb{N}_1)$ and $(w_{\mathrm{E}}, \mathbb{N}_{12})$ are  maximal elements  in $A_{\mathrm{DM}}$.
We shall see, however, that these are not the only maximal elements with respect to $\prec$  (see Section \ref{subsec:Extremal elements}).
Previous studies, such as \cite{GKS21}, have reduced problems on Deligne-Mostow varieties to the ancestral cases. Here we shall generalize this as
\begin{thm}[{Reduction Method  \ref{redmethodsub}}]
It is enough to prove 
Theorem \ref{mainthm:main_extendability} for minimal and maximal elements (with respect to the partial order $\prec$) in $A_{\mathrm{DM}}$. 
\end{thm}

Table \ref{tab:our work on isom} contains a summary of our results.

\renewcommand{\arraystretch}{1.5}
\begin{table}[h]
  \centering
  \caption{Results of this paper. The number in the last row is the number of maximal (resp. minimal) elements with respect to the partial order $\prec$ where (T) holds (resp. does not hold). See Tables \ref{table:DM Gaussian}, \ref{table:DM Eisenstein} in detail.}
   \label{tab:our work on isom}
\begin{tabular}{|c|c|c|c|c|c|}
   \hline
   Number fields & \multicolumn{2}{|c|}{Gaussian} &  \multicolumn{2}{|c|}{Eisenstein} & Total \\ \hline
   Isomorphic? &$M^{\K}_{w,S} \cong X_{w,S}^{\T}$ & $M^{\K}_{w,S} \not\cong X_{w,S}^{\T}$ & $M^{\K}_{w,S} \cong X_{w,S}^{\T}$ & $M^{\K}_{w,S} \not\cong X_{w,S}^{\T}$ & - \\ \hline
      $\#$ Cases  & 16  & 15  & 24  & 30 & 85 \\ \hline
      $\#$  Extremal elements & 2 & 5  & 6  & 21 & 34 \\ \hline
 \end{tabular}
\end{table}

Now, Theorem \ref{mainthm:main_extendability} is, in view of the theory of period maps,  not only interesting in itself, but also has applications in the log minimal model program (LMMP).
Let $\Delta_{w,S}^{\B
B}$ be the boundary divisor $X_{w,S}^{\BB}\setminus M_{w,S}$ with standard coefficients given by the ramification order of the uniformization map $\B^{n-3}\to X_{w,S}$.
Let $\Delta_{w,S}^{\T}$ be the sum of the strict transform of $\Delta_{w,S}^{\BB}$ via the toroidal blow-up $X_{w,S}^{\T}\to X_{w,S}^{\BB}$ and the toroidal boundary with coefficient 1.
Further, let $\Delta_{w,S}^{\K}$ be the sum of the strict transform of $\Delta_{w,S}^{\BB}$ with the above coefficients and the exceptional divisor via the blow-up $M_{w,S}^{\K}\to M_{w,S}^{\mathrm{G}}$. 
See Subsection \ref{subsec:SB method} for the precise description of the discriminant divisors.

\begin{thm}[{\cite[Theorem 5.3]{HKM24}}]
\label{thm:BB_Mumford}
Suppose $(\mathrm{T})$ holds so that the Kirwan blow-up $M_{w,S}^{\K}$ is isomorphic to the toroidal compactification $X_{w,S}^{\tor}$.
Then the following holds in terms of LMMP:
    \begin{enumerate}
    \item The blow-up $(M_{w,S}^{\K}, \Delta_{w,S}^{\K}) \to (X_{w,S}^{\BB}, \Delta_{w,S}^{\BB})$ is log-crepant.
        \item $(X_{w,S}^{\BB}, \Delta_{w,S}^{\BB})$ is the log canonical model and a minimal model of itself.
        \item $(M_{w,S}^{\K}, \Delta_{w,S}^{\K})$ is a log minimal model of itself.
    \end{enumerate}
\end{thm}

In contrast to this, if $(\mathrm{T})$ fails, we obtain the following characterization.

\begin{cor}[{Section \ref{section:applications}}]
\label{maincor:applications}
Let us assume that $(\mathrm{T})$ fails.
In terms of LMMP, the geometry of the Kirwan blow-up $M_{w,S}^{\K}$ can be described as follows:
    \begin{enumerate}
        \item Kirwan's partial resolution $M_{w,S}^{\K}$ is not a semi-toroidal compactification.
    \item In particular, $(M_{w,S}^{\K}, \Delta_{w,S}^{\K})$ is not a log canonical log minimal model.
        \item The two pairs $(M_{w,S}^{\K},\Delta_{w,S}^{\K})$ and $(X_{w,S}^{\T},\Delta_{w,S}^{\T})$ are not log $K$-equivalent.
    \end{enumerate}
\end{cor}
\begin{rem}
\begin{enumerate}
    \item In previous work we proved that $(M_{w,S}^{\K}, \Delta_{w,S}^{\K})$ is not a log canonical pair for the pair $(w_{\mathrm{G}}, \mathbb{N}_{8})$ \cite[Proposition 4.11]{HM25} and  $(w_{\mathrm{E}}, \mathbb{N}_{12})$ \cite[Corollary 5.10]{HKM24}.
    We believe that this also holds for any pair $(w,S)$ not satisfying $(\mathrm{T})$, but at this stage, we cannot exclude the possibility of this being a log canonical, but not log minimal model. 
    \item Also, in previous work \cite{CMGHL23, HKM24, HM25}, the $K$-equivalence of the varieties $M_{w,S}^{\K}$ and $X_{w,S}^{\T}$ (absolutely, and not as pairs) was discussed. 
    This is related to deep results in derived geometry, see \cite[Theorems 1.8, 1.9]{HKM24}. Our approach to out to $K$-equivalene depends on   
     the existence of reflective modular forms, using Borcherds products, which are known for some Deligne-Mostow varieties; see \cite[Section 8]{HKM24}. 
    Due to the lack of general results concerning the existence of reflective modular forms, covering all unitary groups corresponding to the pairs in $A_{\mathrm{DM}}$, this question is currently beyond our reach.
\end{enumerate}
\end{rem}

For the basic notion and definitions of birational geometry, we refer the reader to \cite{KM98}, but also to \cite[Subsection 1.4]{HKM24} for the concepts used in this paper.

\subsection{Organisation of the paper}
In Section \ref{sec:mainresults} we prove the main result Theorem \ref{mainthm:main_extendability}.
For this, we first analyse the intersection behaviour of the strict transform of the discriminant and the boundary divisors. If this is normal crossing (up to finite quotients)  we apply the Borel extension to lift
the Deligne-Mostow map to a morphism between the Kirwan blow-up and the toroidal compactification, which can then be shown to be an isomorphism. In Section \ref{sec:Reduction methods to minimal or maximal varieties}
we introduce a partial ordering among Deligne-Mostow varieties and rove a reduction theorem to minimal or maximal elements in the Deligne-Mostow universe in Theorem \ref{prop:functoriallity_kirwan}.
In a short Section \ref{section:applications} we discuss aspects of LMMP.
Finally, we discuss the structure of the universe of Deligne-Mostow varieties in Section \ref{sec:The moduli spaces of partly ordered points}. 
This mostly serves to illustrate the partial ordering introduced in Section \ref{sec:Reduction methods to minimal or maximal varieties}.

Throughout this paper, all varieties are considered over $\C$.

\subsection*{Acknowledgements}
This research was supported through the program ``Oberwolfach Research Fellows" by the Mathematisches Forschungsinstitut Oberwolfach, which enabled us to visit the institute for an extended stay in October 2023. We thank MFO for the hospitality and the excellent working conditions. The first author was also partially supported by DFG grant Hu 337/7-2. 
This work was also partially supported by the Alexander von Humboldt Foundation through a Humboldt Research Fellowship granted to the second author.

\section{Proof of main results}
\label{sec:mainresults}
We will start with a descriptions of stabilizer groups of strictly semistable points.

\begin{lem}
    \label{lem:stab_sspoints}
    Let $(w,S)$ be a Deligne-Mostow pair and $p\in(\P^1)^n/S[w]\cong (\P^1)^{n-|S|}\times \P^{|S|}$ be a strictly semistable point under the action of $\SL_2(\C)$ with weight given by $w$.
    If $(w,S)$ coincides with one of the three pairs $(w_{\mathrm{G}},\mathbb{N}_8)$ (Gaussian case), $((1/3)^6, \mathbb{N}_6)$ or $(w_{\mathrm{E}},\mathbb{N}_{12})$ (Eisenstein cases), then 
    the stabilizer subgroup $\Stab(p)$ and the connected component of the identity, $\Stab(p)^{\circ}$, are as follows:
    \begin{align*} 
    &\Stab(p)=\Bigl\{\begin{pmatrix}
\lambda & 0 \\
0 & \lambda^{-1} \\
\end{pmatrix}\in\SL_2(\C)\Bigr\}\bigcup \Bigl\{\begin{pmatrix}
0 & \lambda \\
-\lambda^{-1} & 0 \\
\end{pmatrix}\in\SL_2(\C)\Bigr\}\cong \C^{\times}\rtimes \Sigma_2,\\
    &\Stab(p)^{\circ}\cong\C^{\times}.
    \end{align*}    
Otherwise, $\Stab(p)$ is isomorphic to $\C^{\times}$.
\end{lem}
\begin{proof}
Clearly, the groups given are contained in the stabilizer of a strictly semistable point $p$. It remains to show that the stabilizer cannot be bigger. By the characterization of strictly semistable points, the support of $p$
must consist of two points in $\P^1$, which we can take to be $(0:1)$ and $(1:0)$. Let $g\in\Stab(p)$. Then we have two possibilities, namely either that $g$ fixes these points or that $g$ interchanges them. By elementary linear algebra this implies that $g$ is either a diagonal or an anti-diagonal matrix, and since $g \in  \SL_2(\C)$ this shows that \[g= \begin{pmatrix}
\lambda & 0 \\
0 & \lambda^{-1} \\
\end{pmatrix}\  \mathrm{or}\ \begin{pmatrix}
0 & \lambda \\
-\lambda^{-1} & 0 \\
\end{pmatrix}.   \]   

The latter case certainly occurs in the three cases listed in the lemma. Conversely, if the stabilizer contains an element that interchanges the two points in the support of $p$, then this is only possible if all weights are 
equal. By inspection of the list of all arithmetic and non-compact Deligne-Mostow varieties, this only leaves the three cases stated.  
\end{proof}

Now, we discuss the normal crossing property of the discriminant and the boundary. In particular, we want to compare the situation on the Kirwan blow-up with that of the toroidal compactification.
This discussion will also clarify the relationship between $(\mathrm{T})$ and transversality.
\begin{prop}
\label{prop:transversal or non-transversal} For a given pair $(w,S)$, where the weight vector has length $n$, the  following statements hold:
    \begin{enumerate}
    \item The toroidal boundary of the toroidal compactification $X_{w,S}^{\T}$ of the associated ball quotient of $X_{w,S}$ and the strict transform of the discriminant divisor meet generically transversally.
    \item If $(\mathrm{T})$ holds, then the components of the Kirwan exceptional divisor and the strict transform of the discriminant of $M_{w,S}^{\G}$ intersect 
    transversally (up to local finite quotients) everywhere in $M_{w,S}^{\K}$.
    Otherwise $(\mathrm{T})$ fails   
    if and only if there exist components of the exceptional divisor of the Kirwan blow-up  $M_{w,S}^{\K}$ 
    and of the strict transform of the discriminant divisors that meet generically non-transversally.
\end{enumerate}
\end{prop}
\begin{proof}
(1) Let us denote by $\overline{M}_{w,S}$ the Deligne-Mumford compactification of $M_{w,S}$.
The combinatorial description \cite{AL02, Kee92} of the Deligne-Mumford compactification shows that $\overline{M}_{w,S}$ is a normal crossing compactification of $M_{w,S}$.
By \cite[Theorem 1.1]{GKS21}, since the toroidal compactification is compatible with the finite 
group action by $S[w]$, 
the toroidal compactification $X_{w,S}^{\T}$ is isomorphic to the quotient $M_w^{\K}/S[w]$.
In particular, there is a contraction map $f:\overline{M}_{w,S} \to X_{w,S}^{\T}$ described by combinatorial data.
The strict transform of the discriminant divisor and the toroidal boundary appear as certain components of  $\overline{M}_{w,S}\setminus M_{w.S}$ and meet transversally. 
Since the centre of the blow-up $f^{-1}$ has codimension $>2$, this is an isomorphism over the generic point of the intersection of the strict transform of the discriminant divisor and the toroidal boundary in    
$X_{w,S}^{\T}$. Hence transversality in $\overline{M}_{w,S}$ implies generic transversality in $X_{w,S}^{\T}$. This proves (1).

    (2) Let us suppose that the negation of $(\mathrm{T})$ holds, that is, we assume the existence of two sets $T_1$ and $T_2$ satisfying condition (\ref{equ:NT}).
    The computations needed for this part of the proof resemble those of the proofs of  \cite[Propoitipon 3.9]{CMGHL23}, \cite[Theorem 3.4]{HM25} and \cite[Theorem 2.5]{HKM24}. 
    For this reason, we will only give an outline without presenting all technical details.
    The condition on the summation of the weights $w_i$ implies that there exists a polystable point $p$ on $M_{w,S}^{\G}$ whose support consists of two points $\{p_1,p_2\}$ with the following property. 
    All points in the $n$-tuple $p$ which contribute to the cycle supported on $p_1$, have indices in $T_1\sqcup T_2$, whereas all points contributing to $p_2$ have indices in $(T_1\sqcup T_2)^{c}$.    
    Dividing $(\P^1)^n$ by the finite symmetric group $\Sigma_{|S|}$ we obtain the space $((\P^1)^{|S|} \times (\P^1)^{n-|S|})/\Sigma_{|S|} = \P^{|S|} \times (\P^1)^{n-|S|}$. 
    We then have to take the GIT quotient with respect to the induced action of  $\SL_2(\C)$.
    We consider a Luna slice around the polystable point $p$, which we blow up at the origin.
    The \'etale slice theorem asserts that the GIT quotient $\M^{\G}_{w,S}$ around $p$ is locally isomorphic to the quotient of a Luna slice, which is, locally analytically, the germ of the normal bundle of the 
    orbit $\SL_2(\C) \cdot p$, divided by the group $\Stab(p)$. 
    Since $\dim(\Stab (p)) = 1$ by Lemma \ref{lem:stab_sspoints} and $\dim(\SL_2(\C)) = 3$, the Luna slice is locally isomorphic to $\C^{n-2}$.
 
    When $S=T_1$, working with suitable coordinates on $\C^{n-2}$, which we denote by $a_1,\cdots,a_{|S^c|-1}$ and $ b_1,\cdots, b_{|S|-1}$, see the 
    proofs of \cite[Proposition 3.9]{CMGHL23}, \cite[Theorem 3.4]{HM25} and \cite[Theorem 2.5]{HKM24}, a computation shows that in the Luna slice, the discriminant is given by an equation
    \[a_1\cdots a_{|S^c|-1}\cdot \mathrm{disc} (X^{|T_1|} + b_{1} X^{|T_1|-2} + b_{2} X^{|T_1|-3} + \cdots + b_{|T_1|-1}) = 0,\]   
    locally at the origin. 
    If there exists another pair of sets $(T'_1,T'_2)$ with $T_1 \cap T'_1= \emptyset$ contradicting $(\mathrm{T})$, then we obtain a further factor \[\mathrm{disc} (X^{|T'_1|} + c_{1} X^{|T'_1|-2} + c_{2} X^{|T'_1|-3} + \cdots + c_{|T'_1|-1})\]
    to which the same arguments we give below, also apply. 
    Hence, when $T_1\subsetneq S$, taking the coordinates of $\C^{n-2}$ as $a_1,\cdots,a_{|S^c|},b_1,\cdots,b_{|T_1|-1},c_1,\cdots,c_{|S|-|T_1|-1}$, the discriminant is expressed by the equation \[a_1\cdots a_{|S^c|}\cdot \mathrm{disc}(X^{|T_1|} + b_1X^{|T_1|-2} + \cdots + b_{|T_1|-1})\cdot\mathrm{disc}(Y^{|S|-|T_1|} + c_1Y^{|S|-|T_1|-2} + \cdots + c_{|S|-|T_1|-1})=0.\]
    Now, a computer-based calculation of the blow-up shows, as in \cite[Theorem 3.4]{HM25}, that the strict transform (of the component coming from $T_1$) of the discriminant and the exceptional divisor do not meet transversally if and only if $|T_1| \geq 3$.
Note that the possible numbers are $3\leq |T_1| \leq 6$ and the cases of $|T_1| = 4, 6$ were computed in \cite[Theorem 3.4]{HM25}, \cite[Theorem 2.5]{HKM24}, respectively.
The remaining cases $|T_1| = 3, 5$ follow from a similar computation.
This discussion also highlights the meaning of  $(\mathrm{T})$: the existence of $T_1$ is reflected by the existence of an irreducible component of the discriminant divisor defined by the 
versal deformation space of a degeneration worse than $A_1$. 

We finally have to consider the role of elements of finite order in the stabilizer of $p$, noting that though the discriminant divisor we are considering is irreducible, its intersection with the Luna slice is not irreducible, as we saw.
But this finite group action occurs only in three cases and can be checked case by case using the description in  Lemma \ref{lem:stab_sspoints}. 
In all of these cases, the discriminant is given by the equation
\[
\mathrm{disc} (X^{m} + b_{1} X^{m-2} + b_{2} X^{m-3} + \cdots + b_{m-1}) \cdot \mathrm{disc} (Y^{m} + c_{1} Y^{m-2} + c_{2} Y^{m-3} + \cdots + c_{m-1})=0  
\]  
where $m=4,3,6$ respectively.
The stabilizer group of the polystable point $p$ contains a finite group $\Sigma_2$ and this group interchanges the two factors.
The non-transversality claim remains true, since, though the polystable point $p$ has a non-trivial finite stabilizer, a general point on the non-transversal intersection is not fixed by this group action.

Conversely, assuming (T), around all polystable points, the above discussion still applies, where we now have to consider $A_m$-singularities for $m \leq 1$.
Then the same local computations show that the strict transform of the discriminant and the exceptional divisor meet transversally. 
\end{proof}

We shall now give two examples that exemplify the situation. 
\begin{ex}[{Transversal case}]
\label{ex:transversal}
Let us consider the case of $(w_{^G}, \mathbb{N}_2)$.
Since $w_{\G}$ satisfies the INT condition, this pair also fulfills the $\Sigma$INT-$\mathbb{N}_2$ condition.
In this case $(\mathrm{T})$ holds; since $|\mathbb{N}_2| <3$, no subset $T_1 \subset \mathbb{N}_2$ satisfies $|T_1| \geq 3$.
To illustrate how the transversal intersection occurs, we examine a strictly semi-stable point $p$ corresponding to $T_1 = \mathbb{N}_2$ and $T_2 = \{3,4\}$, giving a decomposition $\mathbb{N}_8 = (T_1 \sqcup T_2) \sqcup (T_1 \sqcup T_2)^c$.
As in the proof of Proposition \ref{prop:transversal or non-transversal} (2), the Luna slice is isomorphic to $\C^6$, where the coordinates are denoted by $a_1,\cdots, a_5,b_1$ and the discriminant divisor is given by a normal crossing divisor 
\[a_1\cdots a_5\cdot b_1 =0.\]
Hence, after blowing up at the origin, the strict transform of the discriminant clearly meets the exceptional divisor generically transversally. We also note that we do not have to consider 
any finite group actions in this case. The discussion involving any other strictly semi-stable point is the same.
\end{ex}

\begin{ex}[{Non-transversal case}]
\label{ex:nontransversal}
Let us consider the case of $n=11$ and 
\[w=\left(\frac{1}{3},\frac{1}{6},\frac{1}{6},\frac{1}{6},\frac{1}{6},\frac{1}{6},\frac{1}{6},\frac{1}{6},\frac{1}{6},\frac{1}{6},\frac{1}{6}\right)\]
 which can be found in \cite[Table 2]{GKS21}.
Though this weight does not satisfy the INT condition, taking $S=\{2,\cdots,11\}$, we can check that the pair $(w,S)$ satisfies the $\Sigma$INT-$S$ condition.
Then $(\mathrm{T})$ fails since $|S| \geq 3$ and it suffices to take $T_1 = \{2,3,4,5,6,7\}$ and $T_2 = \emptyset$.
The corresponding strictly semi-stable point $p$ has  coordinates $(\infty,0,0,0,0,0,0,\infty,\infty,\infty,\infty)$, and the description of the discriminant divisor in $\C^9$, with coordinates $a_1,b_1\cdots,b_5,c_1,\cdots, c_3$, is given by
\[a_1 \cdot \mathrm{disc}(X^6 + b_1 X^4 + b_2 X^3 + b_3 X^2 + b_4 X + b_5)\cdot \mathrm{disc}(Y^4 + c_1 Y^2 + c_2 Y + c_3) = 0,\]
 locally around $p$.
As one can see in \cite[Theorem 2.5]{HKM24}, the strict transform of the component $\mathrm{disc}(X^6 + b_1 X^4 + b_2 X^3 + b_3 X^2 + b_4 X + b_5) = 0$ and the exceptional divisor of the Kirwan blow-up 
meet generically non-transversally. The same holds for the component given by $ \mathrm{disc}(Y^4 + c_1 Y^2 + c_2 Y + c_3) = 0$. Again, no finite group actions have to be considered.
Hence we conclude that the Kirwan blow-up is \textit{not} isomorphic to the toroidal compactification, because in the latter space, the strict transform of the discriminant divisor and the toroidal boundary meets transversely up to finite quotients by Proposition \ref{prop:transversal or non-transversal} (1).
\end{ex}

The following theorem is well-known to specialists.
\begin{thm}[{Application of the Borel extension theorem}]
\label{thm:Borel extension}
Let $Y$ be a projective variety and $D$ be a divisor on $Y$. Assume that the following holds: locally in the $\C$-topology, and up to taking finite quotients, $(Y,D)$ is a pair consisting of a smooth variety $Y$ 
and a normal crossing divisor $D$, and assume further that $f: Y \setminus D  \to \B^n / \Gamma$ is a locally liftable map, where $\Gamma$ is an arithmetic unitary group. Then the map $f$ extends to a morphism $\overline{f}: Y \to (\overline{\B^n /\Gamma)}^{\tor}$ to the (unique) toroidal compactification. 
\end{thm}  
\begin{proof}
We first note that extending the map is a local problem and that we can always go to a  Galois cover given by a finite group $G$. It is enough to extend there, as the extended map is $G$-equivariant, and we can take the quotient by $G$. Hence we can assume that $(Y,D)$ is a normal crossing pair. The Borel extension map \cite[Theorem A]{Borel72} says the following: if $\Delta$ is the unit disc and 
$\Delta^*= \Delta \setminus \{0\}$ the punctured  disc, then every holomorphic map $g: (\Delta^*)^a \times \Delta^b \to \B^n\ /\Gamma$ extends to a holomorphic 
map $\overline{g}: \Delta^a \times \Delta^b \to \overline{(\B^n/ \Gamma)}^{\BB}$ to 
the Baily-Borel compactification. 

We now want to lift this map to the toroidal compactification. For this, we use a well-known extension theorem for toroidal compactifications for locally liftable maps, see 
\cite[Theorem 7.29]{Nam80}, \cite[Theorem 7.2]{AMRT10}, \cite[Theorem 5.7]{FC90},
which holds in great generality for locally symmetric domains. Recall that, in general, toroidal compactifications depend on the choice of suitable, compatible fans (for each cusp). 
This theorem says that, at a given cusp, the map can be lifted to the toroidal compactification if the monodromy 
cones defined by $(Y,D)$ are mapped to cones in the fan $\Sigma$ defining the toroidal compactification at this cusp. Since, in the case of ball quotients, these fans only consist of a 1-dimensional ray, 
this condition is empty for ball quotients and hence the map can be extended.  
\end{proof}

We shall now apply this to our situation.  

\begin{thm}
\label{thm:main_extendability}
     The (Deligne-)Mostow isomorphism \eqref{mor:unordered_period_map}
    \[M_{w,S}\hookrightarrow X_{w,S}\defeq\B^{n-3}/\Gamma^{S}_w\]
    extends to an isomorphism between $M^{\K}_{w,S}$ and $X_{w,S}^{\T}$ if and only if $(\mathrm{T})$ is satisfied.
    Otherwise, neither the isomorphism given by \eqref{isom:compactified_sigma_int}, nor its inverse, lift to $M^{\K}_{w,S}$ and $X_{w,S}^{\T}$ respectively. 
    In particular, the Kirwan blow-up and the toroidal compactification are not naturally isomorphic in this case.
\end{thm}
\begin{proof}
  
    Let us first assume that $(\mathrm{T})$ holds. We first claim that the birational map between the Kirwan compactification and  
    the toroidal compactification extends to a morphism. For this we want to check that the conditions of Theorem \ref{thm:Borel extension} apply. Local liftabality of the period map is clear from 
    its construction. The second condition, namely local transversality on the Kirwan side (up to finite quotients), follows from     
    Proposition \ref{prop:transversal or non-transversal} (2). 
    
    Next, we claim that the extended map is an isomorphism. For this, we want to apply Lemma \ref{lem:no_morphism}. 
    It follows from the isomorphism (C$\Sigma$INT) in Theorem \ref{thm:DM_BB} that that the Deligne-Mumford period map extends
    to an isomorphism of the discriminants on the GIT quotient $M^{\G}_{w,S}$ and the Baily-Borel compactification $X_{w,S}^{\T}$.
    Also, the poystable points of the GIT-quotients and the cusps of the Baily-Borel compactification are in $1:1$-correspondence.
    For each polystable point or cusp, the exceptional Kirwan boundary and the toroidal component are irreducible. For the Kirwan blow-up this follows from
    a Luna slice computation 
    and  for the toroidal compactification, this holds by construction. 
    Hence the extended map must map the generic points the exceptional Kirwan divisor bijectively to the generic points of the components of the toroidal boundary (and hence cannot be a divisorial contraction).
    Finally, we note that the toroidal compactification is $\Q$-factorial by construction. Hence Lemma \ref{lem:no_morphism} applies and we get an isomorphism.
    
    Now assume that  $(\mathrm{T})$ fails. By Proposition \ref{prop:transversal or non-transversal} the Kirwan blow-up $M^{\K}_{w,S}$ and the toroidal compactification $X_{w,S}^{\T}$ cannot be naturally isomorphic since 
    we have generic transversality between the boundary and the discriminant on the toroidal side, whereas this fails on the Kirwan side (in each case without taking local finite covers). But then we claim that 
    neither the Deligne-Mumford isomorphism nor its inverse extend to the Kirwan and the toroidal compactification respectively. 
    Indeed, this follows from Lemma \ref{lem:no_morphism} where we use that both  $M^{\K}_{w,S}$ and $X_{w,S}^{\T}$ are $\Q$-factorial. This follows since both varieties have, by construction, at most finite 
    quotient singularities and are hence locally  $\Q$-factorial.   
\end{proof}
In the above proof we have used the following fundamental lemma which gives a useful criterion for certain morphisms to be an isomorphism. 

\begin{lem}
\label{lem:no_morphism}
    Let $X$ and $Y$ be normal projective varieties with $Y$ being $\Q$-factorial and $f:X\dashrightarrow Y$ be a birational map.
    We assume that there exist irreducible divisors $\Delta_{X,1},\cdots, \Delta_{X, m}\subset X$ and $\Delta_{Y,1},\cdots,\Delta_{Y,m}\subset Y$, which are mutually distinct, so that the restriction of $f$ is an isomorphism $f\vert_{X\setminus\Delta_X}:X\setminus\Delta_X\cong Y\setminus\Delta_Y$ where $\Delta_X\defeq\Delta_{X,1}\cup\cdots\cup\Delta_{X,m}$ and $\Delta_Y\defeq\Delta_{Y,1}\cup\cdots\cup\Delta_{Y,m}$.
    If $f$ naturally extends to a morphism $f:X\to Y$ and $f(\Delta_{X,i})\subset\Delta_{Y,i}$, then this must be an isomorphism.
\end{lem}
\begin{proof}
     This follows from a similar discussion in \cite[Theorem 2.7]{HKM24}.
     Since $f$ is a morphism between projective varieties, the image of $X$ in $Y$ is closed.
     In other words, we have $Y=f(X)$ and moreover $f(\Delta_{X, i})=\Delta_{Y, i}$.
     Suppose that $f$ is not an isomorphism.
     This implies that $f$ cannot be bijective since $X$ and $Y$ are normal.
     Hence there exists an $i$ so that $f$ contracts a closed subvariety in $\Delta_{X,i}$, which forces $f$ to be a small contraction.
     This contradicts the fact that $Y$ is $\Q$-factorial, see \cite[Corollary 2.63]{KM98}.
\end{proof}

\section{Reduction methods to minimal or maximal varieties}
\label{sec:Reduction methods to minimal or maximal varieties}
In this section, we study the relationship between various Deligne-Mostow varieties by introducing a suitable partial ordering.
We will see that the proof of Theorem \ref{mainthm:main_extendability} 
can be reduced to considering maximal or minimal elements with respect to this partial order (Reduction Method \ref{redmethodsub}).
\begin{defn}
\label{defn:weights} 
Let $(w, S), (w', S')\in A_{\mathrm{DM}}$ be Deligne-Mostow pairs whose length of the weight vectors are $n$ and $n'$.
We assume (and can do so without any restrictions) that $w_i \leq w_j$ if $i\leq j$, and similarly for $w'$.
    We define a  \textit {partial order} $(w,S)\prec (w',S')$ by the following conditions:
        \begin{itemize}
            \item $n\leq n'$.
            \item $w_i'\leq w_i$ for $1\leq i \leq n$.
            \item $|S|=|S'|$.
            \item $w(S) = w'(S')$.
        \end{itemize}
\end{defn}

\subsection{(Non)Functoriality of the Kirwan blow-up} \label{subsec:functKirwan}
Here, we prove the functoriality of the Kirwan blow-up for Deligne-Mostow pairs satisfying $(w,S)\prec (w',S')$. We will use this later to reduce to the minimal and maximal cases.
Let us first consider the following general situation \cite[Section 3]{Kir85}.
Let $G_i$ be connected reductive groups acting on a smooth projective variety $X_i$ for $i=1,2$ with linearizations $\L_i$.
Let $Z_i$ be a $G_i$-invariant smooth closed irreducible 
subvariety of $X_i$ and take the standard blow-up $\pi_i: \mathrm{Bl}_{Z_i}X_i\to X_i$ along $Z_i$.
We denote by $\mathcal{I}_{Z_i}$ the corresponding ideal sheaf in $\OO_{X_i}$. Assume that $Z_i\cap X_i^{\mathrm{ss}} \neq \emptyset$.
Then the action of $G_i$ on $X_i$ lifts to the standard blow-up $\mathrm{Bl}_{Z_i}(X_i)$ in the ideal sheaf  $\mathcal{I}_{Z_i}$ with linearizations $\widetilde{\L_i}\defeq \pi_i^*\L_i^{\otimes d}\otimes \OO(-E_i)$ by \cite[Section 3]{Kir85}.
Here $d$ is a sufficiently large 
and divisible integer and $E_i$ is the 
exceptional divisor of the blow-up $\pi_i$.
We take GIT-quotients $M(X_i)\defeq X_i/\!/_{\L_i}G_i$ and $M(X_i)^{\K}\defeq \mathrm{Bl}_{Z_i}(X_i)/\!/_{\widetilde{\L_i}}G_i$.
Now the following functoriality holds.
\begin{prop}
\label{prop:functoriallity_kirwan}
Assume that $G_1\subset G_2$ and that there is a closed immersion $\iota: X_1\hookrightarrow X_2$ with $\iota^{-1}(\mathcal{I}_{Z_2}) = \mathcal{I}_{Z_1}$ and $\iota^*\L_2 = \L_1$.
Assume that the group $G_1$ equals  $\Stab_{G_2}(\iota(X_1))$ and that the action of $G_1$ on $X_1$ is equivalent to the action, as a subgroup of $G_2$, on the image $\iota(X_1)$, that is $\iota$ is $G_1$-equivariant.
Finally, assume  that $\iota_{\mathrm{GIT}}:M(X_1)\to M(X_2)$  is a closed immersion.
Then $\iota_{\mathrm{GIT}}$ induces a closed immersion
 \[\iota^{\K}: M(X_1)^{\K}\hookrightarrow M(X_2)^{\K}.\]
 Moreover, if $G\defeq G_1=G_2$ and $\iota_{\mathrm{GIT}}$ is $G$-equivariant, then the induced morphism $\iota^{\K}$ 
 is also $G$-equivariant.
\end{prop}
\begin{proof}
By \cite[Lemma 3.11]{Kir85}, we can identify the (partial) Kirwan blow-up $M(X_i)^{\K}$  
with the  blow-up $\mathrm{Bl}_{(\mathcal{I}_{Z_i}^r)^{G_i}}(M(X_i))$ of the GIT quotient for sufficiently large $r$.
Note that this is not a standard blow-up in general.
Now, by our assumption on the ideal sheaves, for such an $r$, we have $\iota^{-1}(\mathcal{I}_{Z_2}^r) = \mathcal{I}_{Z_1}^r$.
The ideal sheaves $(\mathcal{I}_{Z_i}^r)^{G_i}$ define the closed subschemes $Z_i/\!/G_i$, as in the proof of \cite[Lemma 3.11]{Kir85}.
Since $\iota_{\mathrm{GIT}}$ is a closed immersion, this also holds for  $Z_1/\!/G_1$, 
which forces $\iota^{-1}((\mathcal{I}_{Z_2}^r)^{G_2}) = (\mathcal{I}_{Z_1}^r)^{G_1}$.
Hence, the universality of the blow-up implies that there is a closed immersion $\mathrm{Bl}_{(\mathcal{I}_{Z_1}^r)^{G_1}}(M(X_1))\to \mathrm{Bl}_{(\mathcal{I}_{Z_2}^r)^{G_2}}(M(X_2))$.
\end{proof}
\begin{rem}
    Using this approach, we can, under the assumptions stated, deduce the functoriality of the Kirwan blow-ups immediately, provided there is only one step for the blow-up. This includes the Deligne-Mostow varieties.
    In the general case, it can be achieved by induction, but this needs careful discussion. We would like to emphasize that, in general, the Kirwan blow-up is not functorial.
    In particular, the functoriality of the Kirwan blow-up often fails for finite quotients, as we shall discuss below. 
\end{rem}
Now, we apply this result to our cases.
For pairs $(w,S) \prec (w',S')$, Proposition \ref{prop:partial order} assures the existence of a GIT-theoretical closed immersion
\[\phi_{w,S}^{w',S'}: M_{w,S}^{\mathrm{G}} \hookrightarrow M_{w',S'}^{\mathrm{G}}\]
which is also a closed immersion on the polystable locus.
This implies that these varieties satisfy the assumption in Proposition \ref{prop:functoriallity_kirwan} and hence there is a closed immersion
    \[f^{w',S'}_{w, S}:M_{w,S}^{\K}\hookrightarrow  M_{w',S'}^{\K},\]
lifting $\phi_{w,w'}$.

\begin{rem}\label{rem:finitequotientsK}
Let $(w,S)$ be a Deligne-Mostow pair.
We have a natural finite map 
\[\varphi_w:M_w^{\mathrm{G}}\to M_w^{\mathrm{G}}/S[w]\cong M_{w,S}^{\mathrm{G}}.\]
In this situation, there are Kirwan blow-ups $M_w^{\K}$ and $M_{w,S}^{\K}$ on both sides.
However, we cannot apply Proposition \ref{prop:functoriallity_kirwan} to these spaces because $\varphi_w$ is not a closed immersion (it is obviously not injective).
Nevertheless, combining Theorem \ref{mainthm:main_extendability} with \cite[Theorem 1.1]{GKS21}, if $(w,S)$ satisfies condition (T), then one can verify that there exists a morphism between the corresponding Kirwan blowups, which is compatible with the morphism $\varphi_w$.
By contrast, when condition (T) fails, a similar discussion reveals that no such morphism can exist.
For the case of $(w_{\mathrm{G}},\mathbb{N}_8)$ and $(w_{\mathrm{E}},\mathbb{N}_{12})$, these phenomena were also discussed in \cite{HM25, HKM24}. The case of $8$ points on $\P^1$, for instance, shows that the Kirwan blow-up is, contrary to the 
toroidal compactification, not compatible with the $S_8$-action.
\end{rem}

\subsection{Reduction method}
\label{subsec:SB method}
In this subsection, we will discuss how different Deligne-Mostow varieties whose parameters satisfy the partial order  $(w, S) \prec (w', S')$ are related. 
We start with the relationship between $\prec$,  $(\mathrm{T})$, and GIT quotients. 
\begin{prop}
\label{prop:partial order}
Let us consider two pairs $(w, S), (w', S')\in A_{\mathrm{DM}}$.
    \begin{enumerate}
            \item Assume thar $(w, S) \prec (w', S')$.
        Then, condition $(\mathrm{T})$ holds for $(w,S)$ if and only if it holds for $(w',S')$.
        \item Two pairs are ordered as $(w, S) \prec (w', S')$ if and only if there is a closed immersion
    \[\phi_{w,S}^{w',S'}: M_{w,S}^{\G}\to M_{w',S'}^{\G}\]
    arising naturally from the moduli interpretation  (i.e., the stratification of the collision of points), and taking the quotient of the closed immersion
        \[M_w^{\G}\to M_{w'}^{\G}\]
        by $S[w]$ and $S'[w']$.
    \end{enumerate}
\end{prop}
\begin{proof}
(1) This follows from the fact $|S| = |S'|$ and $w(S) = w'(S')$.

(2) Let us now assume that we have a closed immersion $\phi_{w,S}^{w',S'}$.
This implies that $M_{w,S}^{\G}$ is a stratum where points in $M_{w',S'}^{\G}$ collide.
The dimension of the right-hand side must be at least 
that of the left-hand side, which implies the first condition on $n$ and $n'$, namely $n \leq n'$.
From the perspective of the stratification structure, the conditions of the moduli problem \cite[Section 2]{Has03} imply that the stability condition for $M_{w',S'}^{\G}$ must be 
stronger, that is more restrictive, than for $M_{w,S}^{\G}$.
(Note that a similar argument applies when taking the quotient by the symmetric group; see \cite[Remark 4.4]{KM11} for unordered cases.)
This implies that $w'_i \leq w_i$ for all $i$.
If $|S| >|S'|$, then there is no well-defined  induced map $\phi_{w,S}^{w',S'}$, while if $|S| <|S'|$, then this is not injective. Hence $|S|=|S'|$.
Since $w'_i \leq w_i$, we have in particular that $w(S) \geq w'(S')$.
If $w(S) > w'(S')$, then some ordered points in $M_{w,S}^{\G}$ are mapped to unorderend points in $M_{w',S'}^{\G}$, but this contradicts $|S|=|S'|$.
Conversely assuming the ordering of Deligne-Mostow pairs, from the above discussion, we obtain an inclusion between GIT quotients. This inclusion corresponds to a stratum determined by collisions of points by the weight datum, which in turn induces a closed immersion.
\end{proof}

We define an open subset 
\begin{align}
    \label{defeq:U_w}
U_{w,S}\defeq\{(p_1,\cdots,p_n)\in M_{w,S}^{\G}\mid p_j\neq p_k\ \mathrm{if}\ j,k\in S\}
\end{align}
so that the restriction of $\phi^{w',S'}_{w,S}$ to $U_{w,S}$ is an isomorphism onto its image.
The closed subscheme $D_{w,S}\defeq M_{w,S}^{\G}\setminus U_{w,S}$ generalizes (the closure of) the discriminant divisors in the usual sense.
In fact, the images of the spaces $M_{w,S}^{\G}$ correspond to the collision loci of points on $\P^1$ in the sense of \cite{Dor04, DDH18}.
Let $T_{w,S}^{\K}$ (resp. $T_{w,S}^{\T}$) be the union of the exceptional divisors
 of the blow-up $M_{w,S}^{\K}\to M_{w,S}^{\G}$ (resp. $X_{w,S}^{\T}\to X_{w,S}^{\BB}$).
We denote by $\widetilde{D_{w,S}}^{\K}$ and $\widetilde{D_{w,S}}^{\T}$ the strict transforms of $D_{w,S}$ via these two blow-ups.
Here, by the abuse of notation, $D_{w,S}$ is identified with a divisor on the ball quotient $X_{w,S}^{\BB}$ through the isomorphism (\ref{isom:compactified_sigma_int}).

\begin{lem}[{\cite[Lemma 5.4]{YZ20}}]
\label{lem:YZ}
    Let $f_1:Z_1\to Y$ and $f_2:Z_2\to Y$ be finite morphisms between irreducible algebraic varieties.
    Suppose $Z_1, Z_2$ are normal.
    Moreover, suppose that there exist Zariski-open subsets $U_i$ of $Z_i$, $i=1$ or $2$, with a biholomorphic map $g:U_1\to U_2$, such that $f_1 = f_2 \circ g$.
    Then $g$ extends to an algebraic isomorphism $Z_1\to Z_2$.
\end{lem}

As an application of this lemma, we can obtain the following key method.
   Let $V_{w,S}\subset X_{w,S}$ be the open ball quotient so that we have $U_{w,S}\cong V_{w,S}$, where
   $U_{w,S}$ is the complement of the discriminant divisors as defined in (\ref{defeq:U_w}). 
   
\begin{thm}[{\cite[Theorem 1.1]{GKS21}}]
\label{thm:SB_method}
Let $(w',S')\prec (w,S)$ be Deligne-Mostow pairs.
Assume that there is no isomorphism between $M_{w',S'}^{\K}$ and $X_{w',S'}^{\T}$ extending the Deligne-Mostow isomorphism $M_{w',S'}^{\G}\cong X_{w',S'}^{\BB}$.
Then, the Deligne-Mostow isomorphism $M_{w,S}^{\G}\cong X_{w,S}^{\BB}$ does not  extend to an isomorphism between $M_{w,S}^{\K}$ and $X_{w,S}^{\T}$.
\end{thm}
\begin{proof}
This is very close to the proof of \cite[Theorem 1.1]{GKS21} which works via a reduction to the ancestral cases but for the reader's sake we still thought it helpful to provide some details.   
Suppose that there is an isomorphism $M^{\K}_{w,S}\cong X_{w,S}^{\T}$ commuting with the isomorphism $U_{w,S}\cong V_{w,S}$.
By Proposition \ref{prop:functoriallity_kirwan}, there is a closed immersion $M_{w',S'}^{\mathrm{K}}\hookrightarrow M_{w,S}^{\mathrm{K}}$ compatible with a morphism between GIT quotients.
Hence, Proposition \ref{prop:functoriallity_kirwan} assures the existence of a closed immersion, and in particular a finite morphism, 
$M_{w',S'}^{\K} \to M_{w,S}^{\K}$ compatible with the moduli theoretic morphism $M_{w',S'}^{\G} \to M_{w,S}^{\G}$.
Then, by taking $M_{w',S'}^{\K}$, $X_{w',S'}^{\T}$, $U_{w',S'}$, $V_{w',S'}$ and $M^{\K}_{w,S}\cong X_{w,S}^{\T}$ as $Z_1$, $Z_2$, $U_1$, $U_2$, and $Y$ in Lemma \ref{lem:YZ}, the commutativity of the diagram implies that the isomorphism $M^{\K}_{w,S}\cong X_{w,S}^{\T}$ can be extended to an isomorphism $M^{\K}_{w',S'}\cong X_{w',S'}^{\T}$, making the diagram (Figure \ref{fig:com_diag_sb_proof}) commutative.
This, however, contradicts the assumption for $(w',S')$.
\begin{figure}[h]
\centering
\[
  \begin{tikzpicture}[
    labelsize/.style={font=\scriptsize},
    isolabelsize/.style={font=\normalsize},
    hom/.style={->,auto,labelsize},
    scale=1,
  ]
  
  % nodes
  \node(Uw') at (8,0) {$U_{w',S'}$};
  \node(Vw') at (12,0) {$V_{w',S'}$};
  \node(Mw'K) at (8,4) {$M_{w',S'}^{\K}$};
  \node(Xw't) at (12,4) {$X_{w',S'}^{\T}$};
  \node(Uw) at (0,2) {$U_{w,S}$};
  \node(Vw) at (4,2) {$V_{w,S}$};
  \node(MwK) at (0,6) {$M_{w,S}^{\K}$};
  \node(Xwt) at (4,6) {$X_{w,S}^{\T}$};
  
  % isomorphisms
  \draw[hom,swap] (Uw') to
    node[align=center,yshift=-2pt]{}
    node[swap]{\normalsize $\sim$}
  (Vw');

  \draw[hom] (Uw) to
    node[align=center]{}
    node[swap]{\normalsize$\sim$}
  (Vw);

  % morphisms
  \draw[hom] (Mw'K) to node{} (MwK);
  \draw[hom] (Xw't) to node{} (Xwt);
  % vertical
  \draw[hom,dashed](Mw'K) to node[]{} (Xw't);
  \draw[right hook-latex](Uw') to node[]{} (Mw'K);
  \draw[right hook-latex](Vw') to node[]{} (Xw't);
  \draw[hom,dashed](MwK) to node[]{} (Xwt);
  \draw[right hook-latex](Uw) to node[]{} (MwK);
  \draw[right hook-latex](Vw) to node[]{} (Xwt);
  \draw[right hook-latex] (Uw') to node{} (Uw);
  \draw[right hook-latex] (Vw') to node{} (Vw);
  
  \end{tikzpicture}
\]
\caption{Reduction method}
\label{fig:com_diag_sb_proof}
\end{figure}
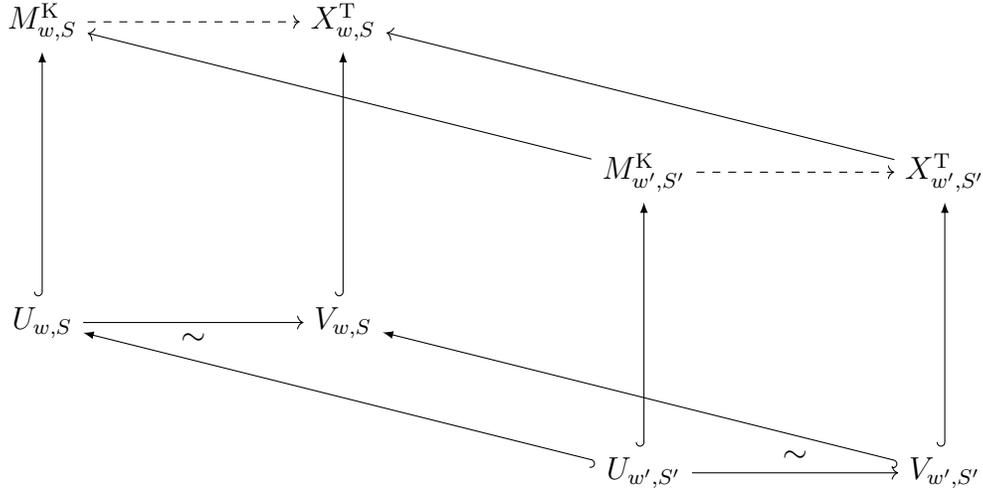
\end{proof}

Theorem \ref{thm:SB_method} opens up a new way of proving Theorem \ref{mainthm:main_extendability}.
Recall that the main question is to decide whether the Deligne-Mostow morphism
$M_{w,S}\to X_{w,S}$ extends to an isomorphism between $M^{\K}_{w,S} \cong X_{w,S}^{\T}$ between the Kirwan blow-up and the toroidal compactification. If this holds, then we say that 
 $(\mathrm{E})$ is satisfied. 
 Theorem \ref{mainthm:main_extendability} now says that this is the case if and only if  $(\mathrm{T})$ holds. This provides an easy numerical criterion, which, moreover is
 preserved under $\prec$. 
 
 In the proof, we considered all Deligne-Mostow varieties at the same time, which required us to distinguish different geometric situations.
 A consequence of Theorem \ref{thm:SB_method} is that one can reduce the proof to the study of extremal cases. Indeed, the theorem says the following: if $(\mathrm{E})$ fails for $(w',S')$ and 
 $(w'S')\ \prec (w,S)$, then  $(\mathrm{E})$ also fails for $(w,S)$. 
Conversely, if  $(\mathrm{E})$ holds for $(w,S)$ and $(w'S')\ \prec (w,S)$, then $(\mathrm{E})$ also holds for $(w',S')$. This can be summarised by: 
\begin{redm}
\label{redmethodsub}
In order to prove Theorem  \ref{mainthm:main_extendability}, it is enough to prove this theorem for \it{extremal}, i.e., \it{minimal} or \it{maximal} Deligne-Mostow varieties. More precisely, let $(w,S)\in A_{\mathrm{DM}}$ be a Deligne-Mostow pair. Then there exists (at least) one minimal pair  $(w_{\min},S_{\min})$ with $(w_{\min},S_{\min}) \prec (w,S)$ and (at least) one  maximal pair $(w_{\max},S_{\max})$ with $(w,S) \prec (w_{\max},S_{\max})$.
If $(\mathrm{E})$ fails for $(w_{\min},S_{\min})$, then it also also fails for $(w,S)$. Similarly, if $(\mathrm{E})$ holds for $(w_{\max},S_{\max})$, then it also holds for $(w,S)$.  
\end{redm}

We close this section with concrete examples, including the cases treated in \cite{GKS21, HKM24, HM25}.

\begin{ex}[{Maximal case}]
Here we shall discuss a case where we can prove the existence of a natural isomorphism between the Kirwan blow-up and the toroidal compactification by reduction to a maximal Deligne-Mostow variety.
For this, we consider the Gaussian case with the data
\[ w\defeq \left(\frac{1}{2},\frac{1}{4},\frac{1}{4},\frac{1}{4},\frac{1}{4},\frac{1}{4},\frac{1}{4}\right),\quad S\defeq \{2,3\}\subset\mathbb{N}_7.\]
Then, $(w, S)\prec (w_{\mathrm{G}}, S')$ where $S'\defeq\{2,3\} \subset \mathbb{N}_8$.
The pair $(w_{\mathrm{G}}, S)$ satisfies  $(\mathrm{T})$ because $|S|=2$ and hence it is impossible to claim the existence of $T_1$ with $|T_1| \geq 3$ (and similarly for $(w, S)$).
Hence, Theorem \ref{thm:SB_method} shows that Theorem \ref{mainthm:main_extendability} for $(w, S)$ follows from the corresponding statement for $(w_{\mathrm{G}}, S)$.
This is the case where (\ref{mor:unordered_period_map}) does extend to an isomorphism between the Kirwan blow-up and the toroidal compactification.
\end{ex}

\begin{ex}[{Minimal case}]
    Let us consider the Eisenstein case with the data
\[ w\defeq \left(\frac{1}{3},\frac{1}{3},\frac{1}{3},\frac{1}{3},\frac{1}{3},\frac{1}{3}\right),\quad S\defeq \{3,4,5,6\}\subset\mathbb{N}_6.\]
Then, we have $(w', S')\prec (w, S)$ 
where 
\[w'\defeq \left( \frac{2}{3},\frac{1}{3},\frac{1}{3},\frac{1}{3},\frac{1}{3}\right),\quad S'\defeq\{2,3,4,5\}\subset \mathbb{N}_5.\]
The pair $(w', S')$ does not satisfy  $(\mathrm{T})$ because we can take $T_1=\{2,3,4\}\subset S$ and $T_2=\emptyset\subset S^c$.
Hence, Theorem \ref{thm:SB_method} implies that Theorem \ref{mainthm:main_extendability} for $(w, S)$ follows from the corresponding statement for $(w', S')$.
Furthermore, we can check that $(w',S')$ is in fact a minimal element in $A_{\mathrm{DM}}$; see Table \ref{table:DM Eisenstein}.
This is the case that (\ref{mor:unordered_period_map}) does not extend to a natural morphism between the Kirwan blow-up and the toroidal compactification.
\end{ex}
    Note that the cases treated in \cite{HM25} (the Gaussian case) and \cite{HKM24} (the Eisenstein case), corresponding to $(w_{\mathrm{G}}, \mathbb{N}_8)$ and $(w_{\mathrm{E}}, \mathbb{N}_{12})$, have no relationship to any other elements in $A_{\mathrm{DM}}$.
    Thus, these pairs can neither be reduced to any other case, nor can any other case be reduced to these pairs.

\begin{rem}
    One of the motivations for this paper originates in the work of Gallardo, Kerr, and Schaffler \cite[Theorem 1.1]{GKS21}. 
    Since moduli of cubic surfaces also have a ball quotient model, one can ask questions similar to those discussed in this paper also in this case.  
    One result of their work says that in the case of the moduli space of cubic surfaces with all 27 lines marked, the toroidal compactification and the Kirwan blow-up coincide \cite[Theorem 1.4]{GKS21}. 
    In contrast, for the ball quotient obtained by forgetting the marking, these two compactifications are known not to be isomorphic \cite{CMGHL23}. As discussed in \cite{DvGK05,CMGHL25}, there is a whole hierarchy of
    moduli spaces of cubic surfaces depending on which partial marking one considers. Conceivably, the methods developed in this paper will also shed new light on these moduli spaces.  
\end{rem}

\begin{rem}
\label{rem:DengGallardo}
Recently, Deng and Gallardo \cite{DG25}
partially extended the classical theory of Deligne-Mostow by constructing eigenperiod maps for general configurations of points on $\mathbb{P}^1$, identifying more general 
conditions under which period maps admit extensions with discrete monodromy.
It would be interesting to investigate how far the ideas developed in this paper can also be applied to this more general setting.
\end{rem}

\section{Applications to birational geometry}
\label{section:applications}
Here, we shall prove Corollary \ref{maincor:applications} as a consequence of Theorem \ref{mainthm:main_extendability} when  $(\mathrm{T})$ fails.
First, we note that the inverse of (\ref{isom:compactified_sigma_int}) does not lift to a morphism between $X_{w,S}^{\T}\to M_{w,S}^{\K}$ by Theorem \ref{mainthm:main_extendability}.
In other words, (a part of) Theorem \ref{mainthm:main_extendability} implies that $M_{w,S}^{\K}$ cannot be located between $X_{w,S}^{\BB}$ and $X_{w,S}^{\T}$ as above.
Hence, it follows from \cite[Theorem 3.23]{AEH21} and \cite[Theorem 5.14]{AE23} that $M_{w,S}^{\K}$ is not a semi-toroidal compactification of $X_{w,S}$.
In this way, we obtain Corollary \ref{maincor:applications} (1).

Now, we can apply  Odaka's theorem \cite[Theorem 3.1]{Oda22}, asserting that a normal compactification of $X_{w,S}$ is a semi-toroidal compactification if and only if it is a log minimal model of itself.
Hence, we conclude that the pair $(M_{w,S}^{\K},\Delta_{w,S}^{\K})$ is not a log canonical log minimal model with a suitable boundary arising from the branch and exceptional divisors of the blow-up.
Arguing in a similar way to  \cite[Corollary 5.8, Theorem 5.9 (2)]{HKM24} this implies that $(M_{w,S}^{\K}, \Delta_{w,S}^{\K})$ and $(X_{w,S}^{\T}, \Delta_{w,S}^{\T})$ are not log $K$-equivalent.
This concludes Corollary \ref{maincor:applications} (2), (3).

\section{On the structure of the universe of Deligne-Mostow varieties}
\label{sec:The moduli spaces of partly ordered points}
Deligne 
and Mostow constructed period maps (\ref{mor:ordered_period_map}) and (\ref{mor:unordered_period_map}) via monodromy of Appell-Lauricella hypergeometric functions.
These maps depend on the numerical datum $n\in\Z$ and   $w\in\Q^n$.
Thurston gave a list in \cite[Appendix]{Thu98}, including cases which had been missed in the original list, see \cite[Proposition 2]{Dor04}.
In particular, the cases of the ball quotients being non-compact and algebraic are listed in \cite[Tables 2, 3]{GKS21}, where arithmetic subgroups arise from Hermitian forms defined either over the rings of Gaussian or Eisenstein integers.
The literature has so far focused primarily on two extremes: 
varieties satisfying the INT condition, which corresponds to fully ordered configurations, and the cases failing the INT condition but satisfying the $\Sigma$INT-$S$ condition, 
where $S$ is taken maximal, which corresponds to the most unordered scenario.
In this paper, we aim to consider all Deligne-Mostow varieties that may arise in between these two extremes, thus including cases that interpolate between full order and complete unordering. We do this 
by considering triples $(w,n,S)$, where $S \subset \mathbb{N}_n$ is a finite set such that $\Sigma$INT-$S$ is satisfied. 
This framework generalizes in a natural way both the INT condition, by taking $S$ to be a singleton, and the $\Sigma$INT-$S$ condition with $S$ maximal, thus providing a unified perspective.

In this section, in order to understand the relationships among these varieties, we provide a comprehensive list of all such cases and analyse how they generalize known examples in the literature.
The maximal cases are closely related to the ancestral cases, which have been studied in some detail in the literature \cite{Dor04, DDH18}, but there are more maximal Deligne-Mostow varieties than the two ancestral cases.

\subsection{The ancestral Deligne-Mostow varieties}
\label{subsection:The ancestral Deligne-Mostow varieties}
Here we recall the ancestral varieties, 
which occupy a special position in the Deligne-Mostow theory. As we saw in the main text, this notion has been generalized as maximal elements by a partial order $\prec$ in our concept, but it provides a concrete example and intuition for our main theorem.
There are two ball quotients, parameterizing ordered 8 points and unordered 12 points:
\begin{align*}
    X_{w_{\mathrm{G}}}^{\BB}&\cong M_{w_{\mathrm{G}}}^{\G}=(\P^1)^8/\!/_{w_{\mathrm{G}}}\SL_2(\C)\\ X_{w_{\mathrm{E}, \mathbb{N}_{12}}}^{\BB}&\cong M_{w_{\mathrm{E}}, \mathbb{N}_{12}}^{\G}=(\P^1)^{12}/\!/_{w_{\mathrm{E}}}\Sigma_{12}\times\SL_2(\C).
\end{align*}
Doran proved that all Deligne-Mostow varieties can be interpreted as a stratum, up to finite group actions, 
of either $X_{w_{\mathrm{G}}}^{\BB}$ or $X_{w_{\mathrm{E}, \mathbb{N}_{12}}}^{\BB}$ depending on whether the lattice is defined over the Gaussian or the Eisenstein integers.
The precise statement of the theorem is as follows.

\begin{thm}[{\cite[Theorem 5, Corollary 9]{Dor04}, \cite[Theorem 4]{DDH18}}]
\label{thm:doran}
Let $X$ be a non-compact Deligne-Mostow variety of arithmetic type in the terminology of \cite[Appendix]{Thu98}.
Then the arithmetic subgroup defining $X$ arises from a lattice over either the Gaussian or the Eisenstein integers and there is a finite morphism from $X$ to a sub-ball quotient of a symmetric group quotient of one of the ancestral Deligne-Mostow 
varieties, either $X_{w_{\mathrm{G}}}^{\BB}$ or $X_{w_{\mathrm{E}},\mathbb{N}_{12}}^{\BB}$.  The image is a locus defined by the property that certain points collide.
\end{thm}
Later, as one consequence of this theorem, we will see that $(w_{\G}, \mathbb{N}_1)$ and $(w_{\mathrm{E}}, \mathbb{N}_{12})$ are indeed maximal elements in $A_{\mathrm{DM}}$ (but not the only ones).
The rest of this paper describes the case of the Gaussian integers exclusively, but a similar discussion also applies to the Eisenstein cases.
In Table \ref{table:details_ordered_Gaussian}, we present the inclusion relations for the Gaussian case explicitly.

\begin{rem}
    We explain the situation in the Gaussian case.
    Let $w_G,w_1,\cdots,w_5$ be the weights listed in \cite[Table 2]{GKS21} from top to bottom.
    We give the full list in Table \ref{table:details_ordered_Gaussian}. Note that in all of these cases, the INT condition holds.

\begin{table}[h]
 \caption{Details of $X_{w_i}^{\BB}$ for the Gaussian cases}
    \label{table:details_ordered_Gaussian}
    \centering
\begin{tabular}{|c|c|c|c|c|} \hline
   Name & Weights & $\dim(X_{w_i}^{\BB})$  & polystable points & stratification \\ \hline
    $w_G$ & $\left(\frac{1}{4}\right)^8$ & 5 & 35 & $(11111111)$ \\ \hline
    $w_1$ & $\left(\frac{1}{2}\right)\left(\frac{1}{4}\right)^6$ & 4 & 15 & $(2111111)$ \\ \hline
    $w_2$ & $\left(\frac{3}{4}\right)\left(\frac{1}{4}\right)^5$ & 3 & 5 & $(311111)$ \\ \hline
    $w_3$ & $\left(\frac{1}{2}\right)^2\left(\frac{1}{4}\right)^4$ & 3 & 7 & $(221111)$ \\ \hline
    $w_4$ & $\left(\frac{3}{4}\right)\left(\frac{1}{2}\right)\left(\frac{1}{4}\right)^3$ & 2 & 3 & $(32111)$ \\ \hline
    $w_5$ & $\left(\frac{1}{2}\right)^3\left(\frac{1}{4}\right)^2$ & 2 & 6 & $(22211)$ \\ \hline
 \end{tabular}
\end{table}

By Theorem \ref{thm:doran}, each $X_{w_i}$ above has a finite morphism to the ancestral Gaussian Deligne-Mostow variety $X_{w_{\mathrm{G}}}$.
These are closed immersions in the case that the INT condition is satisfied. These inclusions are listed in Figure \ref{fig:inclusive_relathionship_Gaussian}.

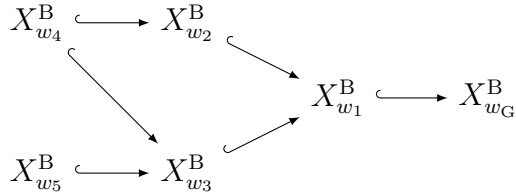
\begin{figure}[h]
\centering
\[
  \begin{tikzpicture}[
    labelsize/.style={font=\scriptsize},
    isolabelsize/.style={font=\normalsize},
    hom/.style={->,auto,labelsize},
    scale=1,
  ]
  
  % nodes
  \node(1) at (4,1) {$X_{w_1}^{\BB}$};
  \node(2) at (2,2) {$X_{w_2}^{\BB}$};
  \node(3) at (2,0) {$X_{w_3}^{\BB}$};
  \node(4) at (0,2) {$X_{w_4}^{\BB}$};
  \node(5) at (0,0) {$X_{w_5}^{\BB}$};
  \node(g) at (6,1) {$X_{w_{\mathrm{G}}}^{\BB}$};

  \draw[right hook-latex] (1) to node[]{} (g);
  \draw[right hook-latex] (4) to node[]{} (2);
  \draw[right hook-latex] (2) to node[]{} (1);
  \draw[right hook-latex] (3) to node[]{} (1);
  \draw[right hook-latex] (5) to node[]{} (3);
  \draw[right hook-latex] (4) to node[]{} (3);

  \end{tikzpicture}
\]
\caption{Inclusions of $X_{w_i}^{\BB}$ in the Gaussian case}
\label{fig:inclusive_relathionship_Gaussian}
\end{figure}
    
\end{rem}
 
\subsection{85 elements in $A_{\mathrm{DM}}$}
\label{subsec:Extremal elements}

In Subsection \ref{subsection:The ancestral Deligne-Mostow varieties}, we reviewed the basic facts about ancestral Deligne-Mostow varieties.
Gallardo, Kerr and Schaffler proved Theorem \ref{thm:GKS21}, reducing the problem to the ancestral cases by using Theorem \ref{thm:doran}.
In this subsection, we give the specific forms of maximal and minimal elements that play a crucial role in Theorem \ref{thm:SB_method} and the Reduction Method \ref{redmethodsub}.
Figure \ref{fig:inclusive_relathionship_Gaussian} shows that if the INT condition is satisfied for Gaussian cases, the ancestral pair $(w_{\G}, \mathbb{N}_1)$ is the unique maximal element.

In the tables below we shall use the notation $\mathbb{N}_{\{i,j\}}\defeq\{i,i+1,\cdots,j\}$.
The relationship between the (original) Gaussian Deligne-Mostow varieties, presented in Table \ref{table:details_ordered_Gaussian} and Figure \ref{fig:inclusive_relathionship_Gaussian}, is generalized to the Deligne-Mostow pairs given in Tables \ref{table:DM Gaussian}, \ref{table:DM Eisenstein}.
There, instead of merely considering the weights $w$, we examine the pairs $(w,S)$, indicating whether they satisfy condition $(\mathrm{T})$, and highlighting those that are either maximal among those satisfying $(\mathrm{T})$, or minimal among those not satisfying it.
Pairs that do not satisfy $(\mathrm{T})$ are marked as NT. Due to space constraints, we present the stratification not in terms of the original weights, but the weights are multiplied by 4 (resp. 6) in the Gaussian (resp. Eisenstein), thus making them integers. 
These weights also encode the stratification relations. 
Elements that are maximal with respect to satisfying $(\mathrm{T})$ are denoted by Max, while minimal elements that fail to satisfy $(\mathrm{T})$ are denoted by Min.
Finally, we can consider the equivalence relation $(w,S)\sim (w',S')$ induced by the partial order $\prec$. By a straight enumeration, we find that there are 10 (respectively 23) equivalence 
classes with respect to $\sim$ in the Gaussian and the  Eisenstein case.

\begin{table}
\caption{Gaussian Deligne-Mostow pairs}
\label{table:DM Gaussian}
\begin{tabular}{|c|c|c|c|c||c|c|c|c|c|}
\hline
$n$ & $4w$ & $S$ & $\mathrm{(T)?}$ &  extremal? & $n$ & $4w$ & $S$ & $\mathrm{(T)?}$ &  extremal? 
\\
\hline
8 & $(11111111)$ & $\mathbb{N}_1$ & T &Max & 6 & $(311111)$ 
&$\mathbb{N}_{\{2,4\}}$ & NT &
\\
\hline
8 & $(11111111)$ & $\mathbb{N}_2$ & T & Max & 6 & $(311111)$ 
&$\mathbb{N}_{\{2,5\}}$ & NT &
\\
\hline
8 & $(11111111)$ & $\mathbb{N}_3$ & NT & &6 & $(311111)$ 
&$\mathbb{N}_{\{2,6\}}$ & NT &Min
\\
\hline
8 & $(11111111)$ & $\mathbb{N}_4$ & NT & & 6 & $(221111)$ &$\mathbb{N}_1$ &T &
\\
\hline
8 & $(11111111)$ & $\mathbb{N}_5$ & NT & & 6 & $(221111)$ &$\mathbb{N}_2$ &T &
\\
\hline
8 & $(11111111)$ & $\mathbb{N}_6$  & NT & & 6 & $(221111)$ &$\mathbb{N}_{\{3,4\}}$ &T &
\\
\hline
8 & $(11111111)$ & $\mathbb{N}_7$& NT &Min & 6 & $(221111)$ &$\mathbb{N}_{\{3,5\}}$ & NT& 
\\
\hline
8 & $(11111111)$ & $\mathbb{N}_8$  & NT & Min & 6 & $(221111)$ &$\mathbb{N}_{\{3,6\}}$ &NT & Min
\\
\hline
7 & $(2111111)$  & $\mathbb{N}_1$& T & & 5 & $(32111)$ &$\mathbb{N}_1$ &T &
\\
\hline
7 & $(2111111)$  & $\mathbb{N}_{\{2,3\}}$  & T & & 5 & $(32111)$ &$\mathbb{N}_{\{3,4\}}$ & T&
\\
\hline
7 & $(2111111)$  & $\mathbb{N}_{\{2,4\}}$& NT & & 5 & $(32111)$ &$\mathbb{N}_{\{3,5\}}$ &T &
\\
\hline
7 & $(2111111)$  & $\mathbb{N}_{\{2,5\}}$ & NT & & 5 & $(22211)$  &$\mathbb{N}_1$ &T &
\\
\hline
7 & $(2111111)$  & $\mathbb{N}_{\{2,6\}}$ & NT& & 5 & $(22211)$  &$\mathbb{N}_{\{1,2\}}$ &T &
\\
\hline
7 & $(2111111)$  & $\mathbb{N}_{\{2,7\}}$ & NT& Min & 5 & $(22211)$  &$\mathbb{N}_{\{1,3\}}$ &T &
\\
\hline
6 & $(311111)$ 
&$\mathbb{N}_1$ &T & & 5 & $(22211)$  &$\mathbb{N}_{\{4,5\}}$ &T  &
\\
\hline
 6 & $(311111)$ 
&$\mathbb{N}_{\{2,3\}}$ &T & & & & & &
\\
\hline
\end{tabular}
\end{table}

\begin{table}
\caption{Eisenstein Deligne-Mostow pairs}
\label{table:DM Eisenstein}
\begin{tabular}{|c|c|c|c|c||c|c|c|c|c|}
\hline
$n$ & $6w$ & $S$ & $\mathrm{(T)?}$ &  extremal? & $n$ & $6w$ & $S$ & $\mathrm{(T)?}$ &  extremal? 
\\
\hline
12 & $(111111111111)$ & $\mathbb{N}_{12}$ & NT & Min & 6 & $(322221)$  &$\mathbb{N}_{\{2,5\}}$ &NT & Min
\\
\hline
11 &$(21111111111)$  & $\mathbb{N}_{\{2,11\}}$ & NT & Min & 6 & $(222222)$  &$\mathbb{N}_1$ &T &Max
\\
\hline
10 &$(3111111111)$ & $\mathbb{N}_{\{2,10\}}$ & NT & Min & 6 & $(222222)$  &$\mathbb{N}_2$ &T & Max
\\
\hline
10 & $(2211111111)$ & $\mathbb{N}_{\{3,10\}}$ & NT & & 6 & $(222222)$  &$\mathbb{N}_3$ & NT& 
\\
\hline
9 & $(411111111)$ & $\mathbb{N}_{\{2,9\}}$ & NT &Min & 6 & $(222222)$  &$\mathbb{N}_4$ &NT &
\\
\hline
9 &$(321111111)$ & $\mathbb{N}_{\{3,9\}}$  & NT & & 6 & $(222222)$  &$\mathbb{N}_5$ &T & Min
\\
\hline
9 & $(222111111)$ & $\mathbb{N}_{\{4,9\}}$& NT & & 6 & $(222222)$  &$\mathbb{N}_6$ & T& Min
\\
\hline
8 & $(22221111)$ & $\mathbb{N}_{\{5,8\}}$  & NT & & 5 & $(42222)$  &$\mathbb{N}_1$ &T &
\\
\hline
8 & $(32211111)$  & $\mathbb{N}_{\{4,8\}}$& NT & & 5 & $(42222)$  &$\mathbb{N}_{\{2,3\}}$ &T &
\\
\hline
8 & $(33111111)$  & $\mathbb{N}_{\{3,8\}}$  & NT &Min & 5 & $(42222)$  &$\mathbb{N}_{\{2,4\}}$  &NT & Min
\\
\hline
8 & $(51111111)$  & $\mathbb{N}_{\{2,8\}}$& NT &Min & 5 & $(42222)$  &$\mathbb{N}_{\{2,5\}}$ &NT & Min
\\
\hline
7 & $(5211111)$  & $\mathbb{N}_{\{3,7\}}$ & NT &Min & 5 & $(33222)$  &$\mathbb{N}_1$  &T  &
\\
\hline
7 & $(4221111)$  & $\mathbb{N}_{\{4,7\}}$ & NT& & 5 & $(33222)$  &$\mathbb{N}_2$  &T  & Max
\\
\hline
7 & $(4311111)$  & $\mathbb{N}_{\{3,7\}}$ & NT&Min & 5 & $(33222)$  &$\mathbb{N}_{\{3,4\}}$  &T  &
\\
\hline
7 & $(3321111)$ 
&$\mathbb{N}_{\{4,7\}}$ &NT & & 5 & $(33222)$  &$\mathbb{N}_{\{3,5\}}$  &NT  & Min
\\
\hline
 7 & $(3222111)$ 
&$\mathbb{N}_{\{5,7\}}$ &NT & & 5 & $(33321)$  &$\mathbb{N}_1$  &T  &
\\
\hline
 7 & $(2222211)$ 
&$\mathbb{N}_{\{6,7\}}$ &T &Max & 5 & $(33321)$  &$\mathbb{N}_2$  &T  &
\\
\hline
 6 & $(531111)$ 
&$\mathbb{N}_{\{3,6\}}$ &NT &Min & 5 & $(33321)$  &$\mathbb{N}_3$  &NT  &Min
\\
\hline
 6 & $(522111)$ 
&$\mathbb{N}_{\{4,6\}}$ &NT &Min & 5 & $(44211)$  &$\mathbb{N}_{\{4,5\}}$  &T  &
\\
\hline
 6 & $(422211)$ 
&$\mathbb{N}_{\{5,6\}}$ &T & & 5 & $(43322)$  &$\mathbb{N}_{\{4,5\}}$  &T  &
\\
\hline
 6 & $(432111)$ 
&$\mathbb{N}_{\{4,6\}}$ &NT &Min & 5 & $(43221)$  &$\mathbb{N}_1$  &T  &
\\
\hline
 6 & $(441111)$ 
&$\mathbb{N}_{\{3,6\}}$ &NT &Min & 5 & $(43221)$  &$\mathbb{N}_{\{3,4\}}$  &T  &
\\
\hline
 6 & $(333111)$ 
&$\mathbb{N}_{\{4,6\}}$ &NT &Min & 5 & $(52221)$  &$\mathbb{N}_1$  &T  &
\\
\hline
 6 & $(332211)$ 
&$\mathbb{N}_{\{5,6\}}$ &T  & & 5 & $(52221)$  &$\mathbb{N}_{\{2,3\}}$  &T  &
\\
\hline
  6 & $(322221)$ 
& $\mathbb{N}_1$ & T  &Max & 5 & $(52221)$  &$\mathbb{N}_{\{2,4\}}$  &NT  & Min
\\
\hline
 6 & $(322221)$ 
&$\mathbb{N}_{\{2,3\}}$ & T &Max & 5 & $(53211)$  &$\mathbb{N}_{\{4,5\}}$  &T  &
\\
\hline
6 & $(322221)$ 
&$\mathbb{N}_{\{2,4\}}$ & NT & & 5 & $(54111)$  &$\mathbb{N}_{\{3,5\}}$  &T  &
\\
\hline
\end{tabular}
\end{table}


\newpage
\begin{thebibliography}{99}

\bibitem[AB12]{AB12}
V. Alexeev, A. Brunyate,
\textit{Extending the Torelli map to toroidal compactifications of Siegel space},
Invent. Math. 188, No. 1, 175-196 (2012).

\bibitem[ABE22]{ABE22}
V. Alexeev, A. Brunyate, P. Engel,
\textit{Compactifications of moduli of elliptic K3
 surfaces: stable pair and toroidal}, 
Geom. Topol. 26, No. 8, 3525-3588 (2022).

\bibitem[AE23]{AE23}
V. Alexeev, P. Engel,
\textit{Compact moduli of K3 surfaces},
Ann. Math. (2) 198, No. 2, 727-789 (2023).

\bibitem[AEH21]{AEH21}
V. Alexeev, P. Engel, C. Han, 
\textit{Compact moduli of K3 surfaces with a nonsymplectic automorphism},
arXiv:2110.13834v2.

\bibitem[ALT+12]{ALT+12}
V. Alexeev, R. Livingston, J. Tenini, M. Arap, X. Hu, L.  Huckaba, P. McFaddin, S. Musgrave, J. Shin, C. Ulrich, 
\textit{Extended Torelli map to the Igusa blow-up in genus 6, 7, and 8}, 
Exp. Math. 21, No. 2, 193-203 (2012).


\bibitem[AMRT10]{AMRT10}
 A. Ash, D. Mumford, M. Rapoport, Y.-S. Tai,
\textit{Smooth compactifications of locally symmetric
varieties},
Cambridge University Press, Cambridge, 2010.

\bibitem[AL02]{AL02}
D. Avritzer, H. Lange,
\textit{The moduli spaces of hyperelliptic curves and binary forms},
Math. Z. 242, No. 4, 615-632 (2002).

\bibitem[Bor72]{Borel72}
A. Borel,
\textit{Some metric properties of arithmetic quotients of symmetric spaces and an extension theorem},
J. Differ. Geom. 6 (1972), 543--560.

\bibitem[CMGHL23]{CMGHL23}
S. Casalaina-Martin, S. Grushevsky, K. Hulek, R. Laza,
\textit{Non-isomorphic smooth compactifications of the moduli space of cubic surfaces}, 
Nagoya Mathematical Journal, 1-51 (2023).

\bibitem[CMGHL25]{CMGHL25}
S. Casalaina-Martin, S. Grushevsky, K. Hulek, R. Laza,
\textit{The birational geometry of moduli of cubic surfaces and cubic surfaces with a line}
Moduli 2, Article ID e1, 24 p. (2025).

\bibitem[DM86]{DM86}
P. Deligne, G.D. Mostow,
\textit{Monodromy of hypergeometric functions and nonlattice integral monodromy},
Publ. Math. IH\'{E}S 63 (1986) 5–89.

\bibitem[DG25]{DG25}
H. Deng, P. Gallardo,
\textit{Eigenperiods and the moduli of points in the line},
Nagoya Mathematical Journal. Published online 2025:1-24.

\bibitem[Dol03]{Dol03}
I. Dolgachev,
\textit{Lectures on invariant theory},
London Mathematical Society Lecture Note Series. 296 (2003).

\bibitem[DvGK05]{DvGK05}
I. Dolgachev, B. van Geemen, S. Kondō, 
\textit{A complex ball uniformization of the moduli space of cubic surfaces via periods of K3
 surfaces}, 
J. Reine Angew. Math. 588, 99-148 (2005).



\bibitem[Dor04]{Dor04}
B. Doran,
\textit{Hurwitz spaces and moduli spaces as ball quotients via pull-back}, arXiv:0404363 (2004).

\bibitem[DDH18]{DDH18}
B. Doran, C. Doran, A. Harder,
\textit{Picard–Fuchs uniformization of modular subvarieties},
Uniformization, Riemann–Hilbert Correspondence, Calabi–Yau Manifolds \& Picard–Fuchs Equations, Adv. Lect.
Math. (ALM), vol. 42, Int. Press, Somerville, MA, 2018, pp. 21–54.

\bibitem[FC90]{FC90}
G. Faltings, C.-L. Chai,
\textit{Degeneration of Abelian Varieties},
Ergeb. Math. Grenzgeb. (3) 22, Springer, Berlin (1990)

\bibitem[GKS21]{GKS21}
P. Gallardo, M. Kerr, L. Schaffler,
\textit{Geometric interpretation of toroidal compactifications of moduli of points in the line and cubic surfaces},
Adv. Math. 381 (2021), Paper No. 107632.

\bibitem[Has03]{Has03}
B. Hassett,
\textit{Moduli spaces of weighted pointed stable curves},
Adv. Math. 173 (2003), no. 2, 316-352.

\bibitem[HKM24]{HKM24}
K. Hulek, S. Kond\={o}, Y. Maeda,
\textit{Compactifications of the Eisenstein ancestral Deligne-Mostow variety}, arXiv:2403.18345 (2024).

\bibitem[HM25]{HM25}
K. Hulek, Y. Maeda,
\textit{Revisiting the moduli space of 8 points on $\P^1$},
Adv. Math. 463 (2025).

\bibitem[Kee92]{Kee92}
S. Keel, 
\textit{Intersection theory of moduli space of stable $n$
-pointed curves of genus zero},
Trans. Am. Math. Soc. 330, No. 2, 545-574 (1992).

\bibitem[KM11]{KM11}
Y.-H. Kiem, H-.B. Moon,
\textit{Moduli spaces of weighted pointed stable rational curves via GIT},
Osaka J. Math. 48 (2011), no. 4, 1115–1140.

\bibitem[KK72]{KK72}
P. Kiernan, S. Kobayashi,
\textit{Satake compactification and extension of holomorphic mappings},
Invent. Math. 16, 237-248 (1972).

\bibitem[Kir85]{Kir85}
F. C. Kirwan,
\textit{Partial desingularisations of quotients of nonsingular varieties and their Betti numbers},
Ann. Math. (2) 122, 41-85 (1985).

\bibitem[KM98]{KM98}
J. Koll{\'a}r, S. Mori,
\textit{Birational geometry of algebraic varieties. {With} the collaboration of {C}. {H}. {Clemens} and {A}. {Corti}}
Cambridge Tracts in Mathematics, 134 (1998), Cambridge University Press.

\bibitem[Laz16]{Laz16}
R. Laza, 
\textit{The KSBA compactification for the moduli space of degree two K3
 pairs}, 
J. Eur. Math. Soc. (JEMS) 18, No. 2, 225-279 (2016).

\bibitem[Loo85]{Loo85}
E. Looijenga, 
\textit{Semi-toroidal partial compactifications I},
Report 8520 (1985), 72 pp., Catholic University Nijmegen.

\bibitem[Loo86]{Loo86} 
E. Looijenga,
\textit{New compactifications of locally symmetric varieties}, Proceedings of the 1984 Vancouver conference in algebraic geometry, 341–364, CMS Conf. Proc., 6, Amer. Math. Soc., Providence, RI, 1986.

\bibitem[Loo03a]{Loo03a}
E. Looijenga,
\textit{Compactifications defined by arrangements. I, - The ball quotient case},
Duke Math. J. 118 (2003), no. 1, 151-187.

\bibitem[Loo03b]{Loo03b}
E. Looijenga, 
\textit{Compactifications defined by arrangements. II, Locally symmetric varieties of Type IV}, 
Duke Math. J. 119 (2003) no.3, 527-588.

\bibitem[MS21]{MS21}
H. Moon, L. Schaffler,
\textit{KSBA compactification of the moduli space of K3
 surfaces with a purely non-symplectic automorphism of order four},
Proc. Edinb. Math. Soc., II. Ser. 64, No. 1, 99-127 (2021).

\bibitem[Mos86]{Mos86}
G.D. Mostow,
\textit{Generalized Picard lattices arising from half-integral conditions}, 
Publ. Math. IH\'ES 63 (1986) 91-106.

\bibitem[Muk12]{Muk12}
S. Mukai
\textit{An introduction to invariants and moduli},
Cambridge Studies in Advanced Mathematics 81. 503 p. (2012).

\bibitem[Nam73]{Nam73}
Y. Namikawa,
\textit{On the canonical holomorphic map from the moduli space of stable curves to the Igusa monoidal transform},
Nagoya Math. J. 52, 197-259 (1973).

\bibitem[Nam76a]{Nam76a}
Y. Namikawa, 
\textit{A new compactification of the Siegel space and degeneration of abelian varieties. I},
Math. Ann. 221, 97-141 (1976).

\bibitem[Nam76b]{Nam76b}
Y. Namikawa, 
\textit{A new compactification of the Siegel space and degeneration of abelian varieties. II},
Math. Ann. 221, 201-241 (1976).

\bibitem[Nam80]{Nam80}
Y. Namikawa, 
\textit{Toroidal Compactification of Siegel Spaces}, 
Lecture Notes in Math. 812, Springer, Berlin (1980)

\bibitem[Oda22]{Oda22}
Y. Odaka, 
\textit{Semi-toroidal and toroidal compactifications as log minimal models, and applications to weak K-moduli},
arXiv:2203.09120.

\bibitem[Sha80]{Sha80}
J. Shah,
\textit{A complete moduli space for K3 surfaces of degree 2},
Ann. Math. (2) 112, 485-510 (1980).

\bibitem[Thu98]{Thu98}
W. P. Thurston,
\textit{Shapes of polyhedra and triangulations of the sphere},
Geom. Topol. Monogr. 1, 511-549 (1998).

\bibitem[YZ20]{YZ20}
C. Yu, Z. Zheng,
\textit{Moduli spaces of symmetric cubic fourfolds and locally symmetric varieties},
Algebra \& Number Theory 14 (10) (2020), 2647-2683.

\end{thebibliography}
\end{document}